\documentclass[a4paper,11pt]{amsart}

\usepackage{mathrsfs}
\usepackage{bm}
\usepackage{color}
\usepackage{multirow,bigdelim}
\usepackage[all]{xy}
\usepackage[driverfallback=dvipdfm]{hyperref}
\usepackage{geometry}
\newtheorem{thm}{Theorem}[section]
\newtheorem{lem}[thm]{Lemma}
\newtheorem{cor}[thm]{Corollary}

\newtheorem{prob}[thm]{Problem}
\newtheorem{claim}[thm]{Claim}

\theoremstyle{definition}
\newtheorem{defin}[thm]{Definition}
\newtheorem{eg}[thm]{Example}
\newtheorem*{conventions}{Conventions}
\newtheorem*{notation}{Notation}
\newtheorem*{acknowledgment}{Acknowledgment}

\theoremstyle{remark}
\newtheorem{rem}[thm]{Remark}

\numberwithin{equation}{section}

\newcommand{\bC}{\mathbb{C}}
\newcommand{\bF}{\mathbb{F}}
\newcommand{\sO}{\mathscr{O}}
\newcommand{\bP}{\mathbb{P}}
\newcommand{\bQ}{\mathbb{Q}}

\newcommand{\bZ}{\mathbb{Z}}

\newcommand{\Sing}{\mathrm{Sing}}
\newcommand{\Supp}{\mathrm{Supp}}

\begin{document}

\title{Minimal compactifications of the affine plane with nef canonical divisors}


\author{}
\address{}
\curraddr{}
\email{}
\thanks{}

\author{Masatomo Sawahara}
\address{Faculty of Education, Hirosaki University, Bunkyocho 1, Hirosaki-shi, Aomori 036-8560, JAPAN}
\curraddr{}
\email{sawahara.masatomo@gmail.com}
\thanks{}

\subjclass[2020]{14J17, 14J26, 14R10, 15A06. }

\keywords{normal analytic surface, minimal compactification of the affine plane, numerically trivial canonical divisor, rational singularity. }

\date{}

\dedicatory{}
\begin{abstract}
We shall consider minimal analytic compactifications of the affine plane with singularities.
In previous work, Kojima and Takahashi proved that any minimal analytic compactification of the affine plane, which has at worse log canonical singularities, is a numerical del Pezzo surface (i.e., a normal complete algebraic surface with the numerically ample anti-canonical divisor) and has only rational singularities. 
In this article, we show that any minimal analytic compactification of the affine plane with the nef canonical divisor has an irrational singularity. 
\end{abstract}
\maketitle
\setcounter{tocdepth}{1}
Throughout this article, we work over the complex number field $\bC$. 
\section{Introduction}\label{1}
Let $X$ be a normal compact analytic surface and let $\Gamma$ be a closed (analytic) curve on $X$. 
Then we say that $(X,\Gamma )$ is a {\it minimal compactification} of the affine plane $\bC ^2$ if $X \backslash \Gamma$ is biholomorphic to $\bC ^2$ and $\Gamma$ is irreducible (see also Definition \ref{def cpt'n}). 
This article is devoted to a discussion of minimal compactifications of the affine plane. 
{\cite{RV60}} proved that if $(X,\Gamma )$ is a minimal compactificaiton of $\bC ^2$ and $X$ is smooth then $X$ is the projective plane $\bP ^2$ and $\Gamma$ is a line on $\bP ^2$. 
Hence, we consider minimal compactifications of the affine plane with singular points in what follows. 

In previous works, minimal compactifications of $\bC ^2$, which are normal hypersurfaces in $\bP ^3$, have been studied by {\cite{Fur97}} and {\cite{Oht01}}. 
Moreover, minimal compactifications of $\bC ^2$ with at most log canonical singularities have been studied. 
More precisely, minimal compactifications of $\bC ^2$ with at most Du Val singularities (resp.\ log terminal singularities, log canonical singularities) were studied by {\cite{MZ88}, etc. (resp.\ {\cite{Koj01b,KT18}}, {\cite{KT09}}) (see also {\cite{Saw22}} for the related work). 
In particular, the following result is summarized in {\cite{KT09}}: 
\begin{thm}[{\cite[Theorem 1.1 (1) and (2)]{KT09}}]\label{KT}
Let $(X,\Gamma )$ be a minimal compactification of $\bC ^2$. 
Assume that $X$ has at worse log canonical singularities and $\Sing (X) \not= \emptyset$. 
Then: 
\begin{enumerate}
\item $X$ is a numerical del Pezzo surface (see \S \S \ref{2-2}, for this definition) of Picard rank one. In particular, $K_X$ is not nef. 
\item All singularities of $X$ are rational. 
\end{enumerate}
\end{thm}
In this article, we consider minimal compactifications of $\bC ^2$ with singularities worse than log canonical singularities. 
Note that for such a minimal compactification of $\bC ^2$ either the anti-canonical divisor is numerically ample or the canonical divisor is nef; indeed, its Picard rank is equal to one. 
Our main result is stated as follows: 
\begin{thm}\label{main}
Let $(X,\Gamma )$ be a minimal compactification of the affine plane $\bC ^2$. 
If the canonical divisor $K_X$ is nef, then $X$ has an irrational singular point. 
\end{thm}
For a normal complete rational algebraic surface $X$, all singularities of $X$ are rational provided that $-K_X$ is numerically ample (see {\cite{Gri98}} or {\cite{Koj01a}}). 
Hence, we obtain the following corollary: 
\begin{cor}
Let $(X,\Gamma )$ be a minimal compactification of the affine plane $\bC ^2$ such that $X$ is a normal complete algebraic surface. 
Then all singularities of $X$ are rational if and only if $X$ is a numerical del Pezzo surface. 
\end{cor}
\subsection*{Organization of article}
In Section \ref{2}, we review definitions of compactifications of the affine plane and prepare some basic notations and facts in order to prove the main result. 
In Section \ref{3}, we deal with some results of minimal compactifications of $\bC ^2$. 
We then refer to {\cite{KT09}} (see also {\cite{Koj01b}}). 
In Section \ref{4}, we observe birational transformations of a compactification $(V,C+D)$ of $\bC ^2$, whose $C$ is a $(-1)$-curve and $D$ contains a branching component, induced by a minimal compactification $(X,\Gamma )$ of $\bC ^2$ by the minimal resolution. 
Especially, assuming that $K_X$ is numerically ample, we show that there exists a birational morphism $\widetilde{f}:V \to \widetilde{V}$ between smooth projective surfaces such that $(\widetilde{V}, \widetilde{C} + \widetilde{D})$ is a compactification of $\bC ^2$ and corresponds to a minimal compactification $(\widetilde{X},\widetilde{\Gamma})$ of $\bC ^2$ with $K_{\widetilde{X}} \equiv 0$ by the contraction of $\widetilde{D}$, where $\widetilde{C}$ is a unique $(-1)$-curve on $\widetilde{f}_{\ast}(D)$ and $\widetilde{D} := \widetilde{f}_{\ast}(D) - \widetilde{C}$. 
In Section \ref{5}, we will prove Theorem \ref{main}. 
In other words, letting $(X,\Gamma )$ be a minimal compactification of $\bC ^2$ with the nef canonical divisor $K_X$, we show that $X$ has an irrational singular point by using arguments in Section \ref{3} and Section \ref{4}. 
We shall briefly outline this proof. 
Let $(V,C+D)$ be a compactification of $\bC ^2$ induced by $(X,\Gamma )$ by the minimal resolution. 
Note that there is a unique connected component $E$ of $D$ with a branching component (see Lemma \ref{lem(3-3)}). 
At first, assuming that $K_X$ is numerically trivial, we find an effective $\bZ$-divisor $Z$ on $V$ such that $\Supp (Z) = \Supp (E)$ and $K_V + Z \equiv 0$; in particular, $p_a(Z) = 1$. 
Meanwhile, assuming that $K_X$ is numerically ample, we construct an effective $\bZ$-divisor $Z$ on $V$ satisfying $\Supp (Z) = \Supp (E)$ and $p_a(Z) = 1$ by using a birational morphism $\widetilde{f}:V \to \widetilde{V}$ as above. 
By these considerations combined with Artin's result ({\cite{Art62}}), we obtain a connected component of $D$, which is contracted to an irrational singular point. 
Finally, Section \ref{6} presents some problems. 
\begin{conventions}
A reduced effective divisor $D$ on a smooth projective surface is called an {\it SNC-divisor} if $D$ has only simple normal crossings. 

For any weighted dual graph, a vertex $\circ$ with the weight $m$ corresponds to an $m$-curve (see also the following Notation).  Exceptionally, we omit this weight (resp.\ we omit this weight and use the vertex $\bullet$ instead of $\circ$) if $m=-2$ (resp.\ $m=-1$). 
\end{conventions}
\begin{notation}
We will use the following notations: 
\begin{itemize}
\item $K_X$: the canonical divisor of a normal surface $X$. 
\item $\varphi ^{\ast}(D)$: the total transform of a divisor $D$ by a morphism $\varphi$. 
\item $\psi _{\ast}(D)$: the direct image of a divisor $D$ by a morphism $\psi$. 
\item $(D \cdot D')$: the intersection number of two divisors $D$ and $D'$ (with Mumford's rational intersection number (cf. {\cite{Mum61,Sak84}}). 
\item $(D)^2$: the self-intersection number of a divisor $D$ (with Mumford's rational intersection number (cf. {\cite{Mum61,Sak84}}). 
\item $D_1 \equiv D_2$: $D_1$ and $D_2$ are numerically equivalent. 
\item $\bF _m$: the Hirzebruch surface of degree $m$, i.e., $\bF _m \simeq \bP (\sO _{\bP ^1} \oplus \sO _{\bP ^1}(m))$. 
\item $m$-curve: a smooth projective rational curve with seif-intersection number $m$. 
\end{itemize}
\end{notation}
\begin{acknowledgment}
The author would like to deeply thank Professor Hideo Kojima for suggesting the problem and for his valuable comments. 
Also, he is grateful to the referee for suggesting many useful comments.
The author is supported by the Foundation of Research Fellows, The Mathematical Society of Japan. 
\end{acknowledgment}
\section{Preliminaries}\label{2}
\subsection{Contractility of divisors on smooth projective surfaces}
Let $V$ be a smooth projective surface, and let $D$ be a connected reduced divisor on $V$. 
We quickly review the contractility condition by analytic means of $D$. 
\begin{defin}
Let the notation be the same as above. 
Then $D$ can be {\it contracted} if there exists a proper surjective morphism $\pi :V \to X$ to a normal analytic surface such that $\pi (D)$ is a point and $V \backslash \Supp (D)$ is isomorphic to $X \backslash \{ \pi (D)\}$ as an analytic space. 
\end{defin}
\begin{lem}\label{cont}
With the same notation as above, $D$ can be contracted if and only if the intersection matrix of $D$ is negative definite. 
\end{lem}
\begin{proof}
This assertion follows from {\cite{Mum61}} and {\cite{Gra62}}. 
\end{proof}
\begin{lem}\label{rational sing}
With the same notation as above, assume further that $D$ can be contracted. 
Then $D$ is contracted to a rational singular point if and only if for every non-zero effective $\bZ$-divisor $Z$ on $V$ with $\Supp (Z) = \Supp (D)$, then $p_a(Z) \le 0$. 
\end{lem}
\begin{proof}
See {\cite{Art62}}. 
\end{proof}
In what follows, assume that every irreducible component of $D$ is a smooth rational curve with self-intersection number $\le -2$ and the intersection matrix of $D$ is negative definite. 
Let $D=\sum _iD_i$ be the decomposition of $D$ into irreducible components. 
Since the intersection matrix of $D$ is negative definite, there exists uniquely a $\bQ$-divisor $D^{\sharp}=\sum _i \alpha _iD_i$ on $V$ such that $(D_i \cdot K_V+D^{\sharp})=0$ for every irreducible component $D_i$ of $D$. 
By {\cite[Lemma 7.1]{Zar62}}, we note that $D^{\sharp}$ is effective. 
\begin{lem}\label{alpha}
With the same notation and the assumptions as above, we have the following assertions: 
\begin{enumerate}
\item Let $D_0$ and $D_1$ be irreducible components of $D$ such that $(D_0 \cdot D_1)=1$. 
If $\alpha _0>0$, then $\alpha _1>0$. 
\item Let $r$ be a positive integer, and let $D_0,D_1,\dots ,D_r$ be irreducible components of $D$ such that $(D_0 \cdot D-D_0)=r$ and $(D_0 \cdot D_i)=1$ for $i=1,\dots ,r$. 
If $\alpha _0,\alpha _1,\dots ,\alpha _{r-1} \in \bZ$, then $\alpha _r \in \bZ$. 
\end{enumerate}
\end{lem}
\begin{proof}
In (1), for simplicity, we put $m_1 := -(D_1)^2 \ge 2$. 
Since $-(D_1 \cdot K_V)=(D_1 \cdot D^{\sharp})$, we then have $-(m_1-2) \ge -m_1\alpha _1 + \alpha _0 > -m_1\alpha _1$. 
Thus, we obtain $\alpha _1 > \frac{m_1-2}{m_1} \ge 0$ by virtue of $m_1 \ge 2$. 

In (2), for simplicity, we put $m_0 := -(D_0)^2$. 
Since $-(D_0 \cdot K_V)=(D_0 \cdot D^{\sharp})$, we then have $-(m_0-2) = -m_0\alpha _0 + \sum _{i=1}^{r}\alpha _i$. 
Here, $\alpha _0,\alpha _1,\dots ,\alpha _{r-1}$ and $m_0$ are integers by the assumption. 
Thus, we know that $\alpha _r \in \bZ$. 
\end{proof}
\subsection{Some notations and properties of compactifications of the affine plane}\label{2-2}
A normal completed surface $X$ is a {\it numerical del Pezzo surface} if the anti-canonical divisor $-K_X$ is numerically ample. 
In other words, $(-K_X)^2>0$ and $(-K_X \cdot \Gamma )>0$ for every curve $\Gamma$ on $X$ with Mumford's rational intersection number
(cf. {\cite{Mum61,Sak84}}). 
Our goal in this article is to show that any minimal compactification of the affine plane with at worse rational singularities is a numerical del Pezzo surface. 
For this purpose, we shall review some definitions and properties of compactifications of $\bC ^2$. 
\begin{defin}\label{def cpt'n}
Let $X$ be a normal compact analytic surface and let $\Gamma$ be a compact analytic curve on $X$. Then: 
\begin{enumerate}
\item The pair $(X,\Gamma )$ is a {\it compactificaion} of the affine plane $\bC ^2$ if $X \backslash \Gamma$ is biholomorphic to $\bC ^2$. 
\item A compactification $(X,\Gamma )$ of the affine plane $\bC ^2$ is {\it minimal} if $\Gamma$ is irreducible. 
\end{enumerate}
\end{defin}
We also recall minimal normal compactifications of $\bC ^2$. 
\begin{defin}
Let $(V,\Delta )$ be a compactification of the affine plane $\bC ^2$ such that $V$ is a smooth projective surface and $\Delta$ is an SNC-divisor on $V$. 
Then we say that $(V,\Delta )$ is a {\it minimal normal compactification} of $\bC ^2$ if $(C \cdot \Delta -C) \ge 3$ for any $(-1)$-curve $C$ on $\Supp (\Delta )$.  
\end{defin}
It is known that minimal normal compactifications of $\bC ^2$ are completely classified by Morrow ({\cite{Mor73}}). 
However, this article will not go into detail because there is no direct use for this classification. 

At the end of this subsection, we shall discuss the canonical divisors of minimal compactifications of the affine plane. 
Let $(X,\Gamma )$ be a minimal compactification of $\bC ^2$, let $\pi :V \to X$ be the minimal resolution, let $D$ be the reduced exceptional divisor of $\pi$, and let $C$ be the proper transform of $\Gamma$ by $\pi$. 
We notice that $(V,C+D)$ is a compactification of $\bC ^2$. 
Let $D=\sum _{i=1}^nD_i$ be the decomposition of $D$ into irreducible components. 
Since the intersection matrix of $D$ is negative definite, there exists uniquely an effective $\bQ$-divisor $D^{\sharp}=\sum _{i=1}^n \alpha _iD_i$ on $V$ such that $(D_i \cdot K_V+D^{\sharp})=0$ for every irreducible component $D_i$ of $D$. 
Then we obtain Lemma \ref{numerical}: 
\begin{lem}\label{numerical}
With the same notation as above, we assume further that $C$ is a $(-1)$-curve. 
Then the following assertions hold: 
\begin{enumerate}
\item The anti-canonical divisor $-K_X$ is numerically ample if and only if $(D^{\sharp} \cdot C)<1$. 
\item The canonical divisor $K_X$ is numerically trivial if and only if $(D^{\sharp} \cdot C)=1$. 
\item The canonical divisor $K_X$ is numerically ample if and only if $(D^{\sharp} \cdot C)>1$. 
\end{enumerate}
\end{lem}
\begin{proof}
We shall prove the assertion (1). 
Since the Picard rank of $X$ is equal to one, we can write $\Gamma \equiv -aK_X$ for some $a \in \bQ _{>0}$. 
Hence, $-K_X$ is numerically ample if and only if $-(K_X \cdot \Gamma )>0$ with Mumford's rational intersection number (see {\cite{Mum61,Sak84}}). 
Moreover, by the definition of Mumford's rational intersection number, we can compute $(K_X \cdot \Gamma ) = (\pi ^{\ast}(K_X) \cdot \pi  _{\ast}^{-1}(\Gamma)) = (K_V+D^{\sharp} \cdot C) = -1+(D^{\sharp} \cdot C)$ because $\pi ^{\ast}(K_X) \equiv K_V + D^{\sharp}$. 
Thus, we obtain the assertion (1). 
Other assertions are the same as (1) and are left to the reader. 
\end{proof}
\begin{rem}
We mainly consider the situation $\rho (V)>2$ in this article. 
Then we will know that $C$ is a $(-1)$-curve by Lemma \ref{lem(3-1)} (3) below. 
\end{rem}
\subsection{Determinants of weighted dual graphs}\label{2-3}
Let $V$ be a smooth projective surface and let $D$ be a reduced divisor on $V$. 
We consider the intersection matrix $I(D)$ of $D$. 
In other words, letting $D=\sum _{i=1}^nD_i$ be the decomposition of $D$ into irreducible components, $I(D) = [(D_i \cdot D_j)]_{1 \le i,\,j \le n}$. 
In this article, for the weighted dual graph $A$ of $D$, we define the {\it determinant} $d(A)$ of $A$ by the determinant of the matrix $-I(D)$. 
Then we shall show the following lemmas: 
\begin{lem}\label{lem(2-3-1)}
Let $(V,\Delta )$ be a compactification of $\bC ^2$ such that $V$ is smooth, and let $A$ be the weighted dual graph of $\Delta$. 
Then we have $d(A)=-1$. 
\end{lem}
\begin{proof}
There exist two birational morphisms $\mu :\bar{V} \to V$ and $\nu :\bar{V} \to \bP ^2_{\bC}$ between smooth projective surfaces such that $V \backslash \Supp (\Delta ) \simeq \bP ^2_{\bC} \backslash \Supp ((\nu \circ \mu ^{-1})_{\ast}(\Delta )) \simeq \bC ^2$. 
Let $B$ be the weighted dual graph of $(\nu \circ \mu ^{-1})_{\ast}(\Delta )$. 
Since $(\nu \circ \mu ^{-1})_{\ast}(\Delta )$ is a line on $\bP ^2_{\bC}$, we obtain $d(B)=-1$. 
On the other hand, since $d(A)=d(B)$ (see\ {\cite[(3.4)]{Fuj82}}), we obtain $d(A)=-1$. 
\end{proof}
\begin{lem}\label{lem(2-3-2)}
Let $D$ and $E$ be two connected reduced divisors on a smooth projective surface $V$ and let $A$ and $B$ be two weighted dual graphs of $D$ and $E$, respectively. 
Assume that there exists a $(-1)$-curve $C$ on $V$ such that $(C \cdot D) = (C \cdot E)=1$ and $C+D+E$ can be contracted to a single $(-1)$-curve (on a smooth projective surface). 
Then $d(A)$ and $d(B)$ are prime to each other. 
\end{lem}
\begin{proof}
By the assumption and {\cite[(3.4)]{Fuj82}}, there exist two integers $a$ and $b$ such that we have: 
\begin{align}\label{det}
\begin{split}
1 = \left| \begin{array}{ccc}
d(A) & -a & 0 \\
-1 & 1 & -1 \\
0 & -b & d(B) 
\end{array} \right|
&= d(A)d(B) - b \cdot d(A) -a \cdot d(B) \\
&= (d(B)-b) d(A) + (-a) d(B).
\end{split}
\end{align}
This implies that $d(A)$ and $d(B)$ are prime to each other. 
\end{proof}
\section{Boundaries of minimal compactifications of the affine plane}\label{3}
In this section, we summarize some results on minimal compactifications of $\bC ^2$, mainly based on {\cite{KT09}} (see also {\cite{Koj01b}}). 

Let $(X,\Gamma )$ be a minimal compactification of $\bC ^2$ such that $\Sing (X) \not= \emptyset$, let $\pi :V \to X$ be the minimal resolution of $X$, let $D$ be the reduced exceptional divisor of $\pi$, and let $C$ be the proper transform of $\Gamma$ by $\pi$. 
Then we obtain a compactification $(V,C+D)$ of $\bC ^2$. 

The following facts (Lemmas \ref{lem(3-1)}, \ref{lem(3-2)} and \ref{lem(3-3)}) refer to {\cite[Lemmas 4.1--4.6]{KT09}}. 
They are shown by the computation of birational transformations between smooth projective surfaces and the classification of minimal normal compactifications of the affine plane (see {\cite{Mor73}}). 
As a remark, {\cite{KT09}} assumes that $X$ has at most log canonical singularities, but results {\cite[Lemmas 4.1--4.6]{KT09}} do not use this assumption. 
Hence, the following Lemma \ref{lem(3-1)} be shown immediately by referring to the proof of {\cite[Lemmas 4.1--4.4]{KT09}}, so we omit the proof of this lemma in this article: 
\begin{lem}[{\cite[Lemmas 4.1--4.4]{KT09}}]\label{lem(3-1)}
With the same notation as above, we have: 
\begin{enumerate}
\item $\sharp \Sing (X) \le 2$. In other words, $D$ consists of at most two connected components. 
\item $C+D$ is an SNC-divisor. 
\item If $C$ is not a $(-1)$-curve, then $V=\bF _m$ for some $m \ge 2$. 
\end{enumerate}
\end{lem}
If $X$ has at most log canonical singularities, then the minimal compactification $(X,\Gamma )$ of $\bC ^2$ has already been studied by {\cite{Koj01b,KT09}}. 
In what follows, assume that $X$ has a singular point worse than log canonical singularities. 
Then $C$ is a $(-1)$-curve by Lemma \ref{lem(3-1)} (3) because the Picard rank of $V$ is greater than two. 
Let $\nu :V \to W$ be a sequence of contractions of $(-1)$-curves and subsequently (smoothly) contractible curves in $\Supp (D)$, starting with the contraction of $C$, such that the pair $(W,D_W)$ is a minimal normal compactification of $\bC ^2$, where $D_W := \nu _{\ast}(C+D)$.
\begin{lem}\label{lem(3-2)}
With the same notation and assumptions as above, we may assume the following two conditions: 
\begin{itemize}
\item $(W,D_W) = (\bF _m,M+F)$ for some $m \ge 2$, where $M$ and $F$ are the minimal section and a fiber of the $\bP ^1$-bundle $\bF _m \to \bP ^1_{\bC}$, respectively. 
\item The point $\nu (C)$ does not lie on $M$. 
\end{itemize}
\end{lem}
\begin{proof}
We write $V = V_0 \overset{f_1}{\to} V_1 \overset{f_2}{\to} \dots \overset{f_{\ell}}{\to} V_{\ell} = W$ $(\ell \in \bZ _{>0})$ as a composition of blowing-downs $f_i$ at a $(-1)$-curve for $i=1,\dots ,\ell$. 
Moreover, we set $\nu _i := f_i \circ \dots \circ f_1$ for $i=1,\dots ,\ell$. 
By {\cite[Lemma 4.5]{KT09}}, we know that $(W,D_W) = (\bP ^2_{\bC},L)$ or $(\bF _m,M+F)$, where $L$ is a line on $\bP ^2_{\bC}$. 

Assume that $(W,D_W) = (\bP ^2_{\bC},L)$. 
Then we know that $V_{\ell -1} = \bF _1$ and $\nu _{\ell -1,\ast}(C+D)$ consists of the minimal section $M_1$ and the fiber $F_1$, which is $f_{\ell ,\ast}^{-1}(L)$, of the $\bP ^1$-bundle $\bF _1 \to \bP ^1_{\bC}$. 
We note $\ell \ge 2$ because $V \not= \bF _1$. 
Moreover, we notice that the point $\nu _{\ell -1}(C)$ lies on the intersection $M_1 \cap F_1$. 
Indeed, otherwise, we know that $C+D$ contains at least two rational curves with self-intersection number $\ge -1$ by a similar argument to {\cite[Lemma 4.5]{KT09}}, so that we obtain a contradiction. 
Hence, $V_{\ell -2}$ is a weak del Pezzo surface of degree $7$ with only one $(-2)$-curve $f^{-1}_{\ell -1,\ast}(M_1)$. 
Moreover, $V_{\ell -2}$ contains exactly two $(-1)$-curves $f^{-1}_{\ell -1,\ast}(F_1)$ and the other, say $C_{\ell -2}$. 
Then we know that $\nu _{\ell -2,\ast}(C+D)=f^{-1}_{\ell -1,\ast}(M_1)+f^{-1}_{\ell -1,\ast}(F_1)+C_{\ell -2}$ and $\nu _{\ell -2}(C) \in f^{-1}_{\ell -1,\ast}(F_1) \cap C_{\ell -2}$. 
Hence, by replacing $f_{\ell} \circ f_{\ell -1}$ to the contraction $V_{\ell -2} \to \bF _2$ of $f_{\ell -1,\ast}^{-1}(F_1)$, we can assume that $(W,D_W)=(\bF _2,M+F)$ and $\nu (C) \not\in M$. 

Assume that $(W,D_W) = (\bF _m,M+F)$ for some $m \ge 2$ but $\nu (C) \in M$. 
Then $V_{\ell -2}$ contains exactly two $(-1)$-curves $f^{-1}_{\ell ,\ast}(F)$ and the other, say $C_{\ell -1}$. 
By the configurations, we know that $\nu _{\ell -1,\ast}(C+D)=f^{-1}_{\ell ,\ast}(M)+f^{-1}_{\ell ,\ast}(F)+C_{\ell -1}$ and $\nu _{\ell -1}(C) \in f^{-1}_{\ell ,\ast}(F) \cap C_{\ell -1}$. 
Hence, by replacing $f_{\ell}$ to the contraction $V_{\ell -1} \to \bF _{m+1}$ of $f_{\ell ,\ast}^{-1}(F)$, we can assume that $(W,D_W)=(\bF _{m+1},M+F)$ such that $\nu (C) \not\in M$. 
\end{proof}
Hence, we assume that $(W,D_W)$ satisfies the assumption in Lemma \ref{lem(3-2)}. 
By applying an argument of {\cite[Lemma 4.6]{KT09}}, we then obtain the following lemma: 
\begin{lem}\label{lem(3-3)}
With the same notation and assumptions as above, the dual graph of $C+D$ is as follows: 
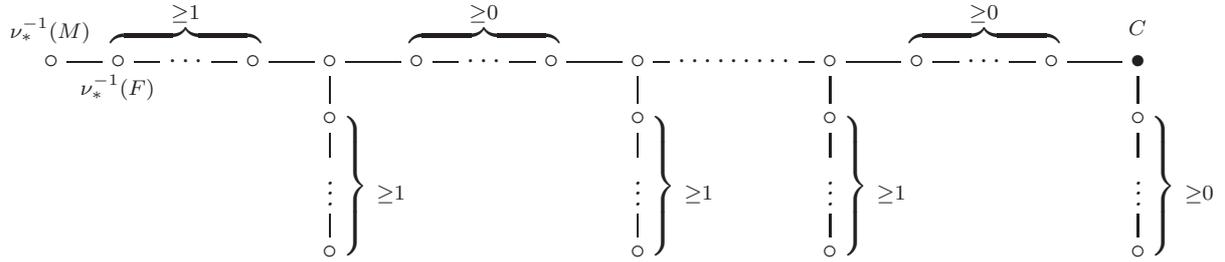
\begin{align*}
\xygraph{
\circ ([]!{+(0,.3)} {^{\nu _{\ast}^{-1}(M)}}) - []!{+(.7,0)} \circ ([]!{+(0,-.3)} {^{\nu _{\ast}^{-1}(F)}})
- []!{+(.7,0)} \cdots ([]!{+(0,.4)} {\overbrace{\quad \qquad \qquad}^{\ge 1}}) - []!{+(.7,0)} \circ - []!{+(.8,0)} \circ
(- []!{+(0,-.6)} \circ - []!{+(0,-.7)} \vdots ([]!{+(.55,0)} {\left\} \begin{array}{l} \ \\ \ \\ \! ^{\ge 1} \\ \ \end{array} \right.}) - []!{+(0,-.7)} \circ,
- []!{+(.9,0)} \circ - []!{+(.7,0)} \cdots ([]!{+(0,.4)} {\overbrace{\quad \qquad \qquad}^{\ge 0}}) - []!{+(.7,0)} \circ - []!{+(.9,0)} \circ
(- []!{+(0,-.6)} \circ - []!{+(0,-.7)} \vdots ([]!{+(.55,0)} {\left\} \begin{array}{l} \ \\ \ \\ \! ^{\ge 1} \\ \ \end{array} \right.}) - []!{+(0,-.7)} \circ,
-[r] {\cdots \cdots \cdots} -[r] \circ
(- []!{+(0,-.6)} \circ - []!{+(0,-.7)} \vdots ([]!{+(.55,0)} {\left\} \begin{array}{l} \ \\ \ \\ \! ^{\ge 1} \\ \ \end{array} \right.}) - []!{+(0,-.7)} \circ,
- []!{+(.9,0)} \circ - []!{+(.7,0)} \cdots ([]!{+(0,.4)} {\overbrace{\quad \qquad \qquad}^{\ge 0}}) - []!{+(.7,0)} \circ - []!{+(.9,0)} \bullet ([]!{+(0,.3)} {^C})
- []!{+(0,-.6)} \circ - []!{+(0,-.7)} \vdots ([]!{+(.55,0)} {\left\} \begin{array}{l} \ \\ \ \\ \! ^{\ge 0} \\ \ \end{array} \right.}) - []!{+(0,-.7)} \circ
)))
}
\end{align*}
Here, weights of all vertices of the above graph are omitted. 
\end{lem}
\begin{proof}
By the construction of $\nu$, we know that $\nu$ is the sequence of blow-ups at a point on $F \backslash (M \cap F)$ and its infinitely near points. 
In particular, $\nu$ is a composite of Euclidean transformations and EM transformations (see {\cite[p.\ 92]{Miy78}} and {\cite[p.\ 100]{Miy78}}, for these definitions). 
Thus, we obtain this assertion. 
\end{proof}
\section{Birational transformations of minimal compactifications of the affine plane}\label{4}
Let $(X,\Gamma )$ be a minimal compactification of the affine plane $\bC ^2$, let $\pi :V \to X$ be the minimal resolution, let $D$ be the reduced exceptional divisor of $\pi$, and let $C$ be the proper transform of $\Gamma$ by $\pi$. 
Then we obtain a compactification $(V,C+D)$ of the affine plane $\bC ^2$. 
Assume that $X$ has a non-cyclic quotient singular point, i.e., $D$ admits a branching component. 
Hence, $C$ is a $(-1)$-curve by Lemma \ref{lem(3-1)} (3). 
In particular, we obtain the contraction $f:V \to V'$ of $C$. 
Then there exists uniquely a $(-1)$-curve $C'$ on $\Supp (f_{\ast}(D))$ by Lemma \ref{lem(3-1)} (3). 
Put $D' := f_{\ast}(D)-C'$. 
Thus, $(V',C'+D')$ is a compactification of $\bC ^2$. 

Let $D=\sum _iD_i$ be the decomposition of $D$ into irreducible components. 
Since the intersection matrix of $D$ is negative definite, there exists uniquely an effective $\bQ$-divisor $D^{\sharp}=\sum _i \alpha _iD_i$ on $V$ such that $(D_i \cdot K_V+D^{\sharp})=0$ for every irreducible component $D_i$ of $D$. 
Similarly, if the intersection matrix of $D'$ is negative definite, there exists uniquely an effective $\bQ$-divisor ${D'}^{\sharp}=\sum _{i'} \alpha _{i'}'D_{i'}'$ on $V'$ such that $(D_{i'}' \cdot K_{V'}+{D'}^{\sharp})=0$ for every irreducible component $D_{i'}'$ of $D'$. 

In this section, we study the birational morphism $f:(V,C+D) \to (V',C'+D')$; in particular, the condition that the intersection matrix of $D'$ is negative definite. 
From now on, we consider the following three cases separately. 
\subsection{Case(1)}\label{4-1}
In this subsection, with the same notation at the beginning of \S \ref{4}, assume further that $D$ is connected. 
Then $(C \cdot D)=1$ and $(C' \cdot D') \le 2$. 
More precisely, $C$ meets exactly one irreducible component of $D$, say $D_1$, and $C'$ meets at most two irreducible components of $D'$. 
Then we notice $C' = f_{\ast}(D_1)$. 
In particular, $D_1$ is a $(-2)$-curve. 
Let $A$ and $B$ be weighted dual graphs of connected components of $D'-C'$, where $B = \emptyset$ and $d(B) := 1$ if $(C' \cdot D')=1$. 
We may assume that $A$ contains a branching component. 
Then the weighted dual graphs of $C+D$ and $C'+D'$ are as follows: 
\begin{align*}
\xygraph{A -[r] \circ ([]!{+(0,.2)} {^{D_1}}) -[r] \bullet ([]!{+(0,.2)} {^C})}
\quad {\overset{f}{\longrightarrow}} \quad
\xygraph{A -[r] \bullet ([]!{+(0,.2)} {^{C'}})}
&\quad \text{if $(C' \cdot D')=1$}, \\
\xygraph{\bullet ([]!{+(0,.2)} {^C}) -[l] \circ ([]!{+(0,.2)} {^{D_1}}) (-[]!{+(-1,.3)} A,-[]!{+(-1,-.3)} B)}
\quad {\overset{f}{\longrightarrow}} \quad
\xygraph{\bullet ([]!{+(0,.2)} {^{C'}}) (-[]!{+(-1,.3)} A,-[]!{+(-1,-.3)} B)}
&\quad \text{if $(C' \cdot D')=2$}. 
\end{align*}
Let $D_1'$ and $D_2'$ be irreducible components of $D'$ meeting $C'$, where we consider $D_2':=0$ if $(D' \cdot C')=1$. 
Let $\underline{A}$ and $\underline{B}$ be weighted dual graphs of two connected components of $D'-(D_1'+D_2')$ such that $\underline{A}$ is a subgraph of $A$, where we consider $d(\underline{B}):=0$ (resp.\ $d(\underline{B}):=1$) if $B = \emptyset$ (resp.\ $B \not= \emptyset$ and $\underline{B} = \emptyset$). 
Then we shall observe the relationship between $d(A)$, $d(B)$, $d(\underline{A})$ and $d(\underline{B})$. 
We first notice that $C'+D'$ is a boundary of a compactification of $\bC ^2$. 
By Lemma \ref{lem(2-3-1)}, we have: 
\begin{align}\label{(4-1(1))}
-1 = \left| \begin{array}{ccc}
1 & -1 & -1\\
-d(\underline{A}) & d(A) & 0 \\
-d(\underline{B}) & 0 & d(B) \\
\end{array} \right| = d(A)d(B) - d(A)d(\underline{B}) - d(\underline{A})d(B). 
\end{align}
Moreover, as the intersection matrix of $D$ is negative definite, we obtain: 
\begin{align}\label{(4-1(2))}
\left| \begin{array}{ccc}
2 & -1 & -1\\
-d(\underline{A}) & d(A) & 0 \\
-d(\underline{B}) & 0 & d(B) \\
\end{array} \right| {\underset{(\ref{(4-1(1))})}{=}} d(A)d(B)-1 >0. 
\end{align}

By the above considerations, we obtain the following lemma: 
\begin{lem}\label{lem(4-1)}
We have the following assertions: 
\begin{enumerate}
\item The intersection matrix of $D'$ is negative definite. 
\item If $(C \cdot D^{\sharp})>1$, then $(C' \cdot {D'}^{\sharp}) \ge 1$. 
\item If $(C \cdot D^{\sharp}) \le 1$, then $(C' \cdot {D'}^{\sharp})<1$. 
\end{enumerate}
\end{lem}
\begin{proof}
In (1), this assertion is obvious because the weighted dual graph of $D'$ is a subgraph of the weighted dual graph of $D$. 

In (2) and (3), since the intersection matrix of $D'$ is negative definite by (1), there exists uniquely an effective $\bQ$-divisor ${D'}^{\sharp}=\sum _{i'} \alpha _{i'}D_{i'}'$ on $V'$ such that $(D_{i'}' \cdot K_{V'}+{D'}^{\sharp})=0$ for every irreducible component $D_{i'}'$ of $D'$. 
Let $\alpha _1$ be the coefficient of $D_1$ of $D^{\sharp}$ and let $\alpha _1'$ and $\alpha _2'$ be coefficients of $D_1'$ and $D_2'$ of ${D'}^{\sharp}$, respectively, where we consider $\alpha _2' := 0$ if $(C' \cdot D')=1$. 
When $(C' \cdot D')=2$, we may assume that the weighted dual graphs $A$ and $B$ correspond to connected components of $D'$ containing $D_1'$ and $D_2'$, respectively. 
By the Cramer formula (cf.\ (\ref{B0})), we can write: 
\begin{align*}
\alpha _1 = \frac{a}{d(A)d(B)-1},\ \alpha _1' + \alpha _2' = \frac{a_1'}{d(A)} + \frac{a_2'}{d(B)} = \frac{a_1'd(B)+a_2'd(A)}{d(A)d(B)}
\end{align*}
for some positive integers $a$, $a_1'$ and $a_2'$, where $a_2'=0$ if $(C' \cdot D')=1$. 
Since $D_1$ is a $(-2)$-curve, we further obtain $a=a_1'd(B)+a_2'd(A)$ by using the cofactor expansion of the determinant. 
Meanwhile, we know that $(C \cdot D^{\sharp}) = \alpha _1$ and $(C' \cdot {D'}^{\sharp}) = \alpha _1'+\alpha _2'$. 
Thus, we obtain assertions (2) and (3) because $d(A)d(B)$ and $a$ are positive integers. 
\end{proof}
\subsection{Case(2)}\label{4-2}
In this subsection, with the same notation at the beginning of \S \ref{4}, assume further that $D'$ is connected but $D$ is not. 
Then $(C \cdot D)=2$ and $(C' \cdot D')=1$. 
More precisely, $C$ meets exactly two irreducible components of $D$, say $D_1$ and $D_2$, and $C'$ meets exactly one irreducible component of $D'$, say $D_1'$. 
Here, we may assume that $D_1' = f_{\ast}(D_1)$. 
Then we notice $C' = f_{\ast}(D_2)$. 
In particular, $D_1$ and $D_2$ are a $(-3)$-curve and a $(-2)$-curve, respectively. 
Let $g:V' \to V''$ be a contraction of $C'$. 
Since $D$ admits a branching component, so does $(g \circ f)_{\ast}(D) = g_{\ast}(D')$ by construction of $f$ and $g$. 
Hence, there exists uniquely a $(-1)$-curve $C''$ on $\Supp (g_{\ast}(D'))$ by Lemma \ref{lem(3-1)} (3) and the configuration of $C'+D'$. 
Put $D'' := g_{\ast}(D')-C''$. 
Thus, $(V'',C''+D'')$ is a compactification of $\bC ^2$. 
Let $A$ and $B$ be weighted dual graphs of two connected components of $D''-C''$, where $B = \emptyset$ and $d(B) := 1$ if $(C'' \cdot D'')=1$. 
We may assume that $A$ contains a branching component. 
Then the weighted dual graphs of $C+D$, $C'+D'$ and $C''+D''$ are as follows: 
\begin{align*}
\xygraph{\circ ([]!{+(0,.2)} {^{D_2}}) -[l] \bullet ([]!{+(0,.2)} {^C}) -[l] \circ ([]!{+(0,.2)} {^{D_1}}) ([]!{+(0,-.3)} {^{-3}}) -[l] A}
\quad {\overset{f}{\longrightarrow}} \quad
\xygraph{\bullet ([]!{+(0,.2)} {^{C'}}) -[l] \circ ([]!{+(0,.2)} {^{D_1'}}) -[l] A}
\quad {\overset{g}{\longrightarrow}} \quad
\xygraph{\bullet ([]!{+(0,.2)} {^{C''}}) -[l] A}
&\quad \text{if $(C'' \cdot D'')=1$}, \\
\xygraph{\circ ([]!{+(0,.2)} {^{D_2}}) -[l] \bullet ([]!{+(0,.2)} {^C}) -[l] \circ ([]!{+(0,.2)} {^{D_1}}) ([]!{+(0,-.3)} {^{-3}}) (-[]!{+(-1,.3)} A,-[]!{+(-1,-.3)} B)}
\quad {\overset{f}{\longrightarrow}} \quad
\xygraph{\bullet ([]!{+(0,.2)} {^{C'}}) -[l] \circ ([]!{+(0,.2)} {^{D_1'}}) (-[]!{+(-1,.3)} A,-[]!{+(-1,-.3)} B)}
\quad {\overset{g}{\longrightarrow}} \quad
\xygraph{\bullet ([]!{+(0,.2)} {^{C''}}) (-[]!{+(-1,.3)} A,-[]!{+(-1,-.3)} B)}
&\quad \text{if $(C'' \cdot D'')=2$}. 
\end{align*}
Let $D_1''$ and $D_2''$ be irreducible components of $D''$ meeting $C''$, where we consider $D_2'':=0$ if $(D'' \cdot C'')=1$. 
Let $\underline{A}$ and $\underline{B}$ be weighted dual graphs of two connected components of $D''-(D_1''+D_2'')$ such that $\underline{A}$ is a subgraph of $A$, where we consider $d(\underline{B}):=0$ (resp.\ $d(\underline{B}):=1$) if $B = \emptyset$ (resp.\ $B \not= \emptyset$ and $\underline{B} = \emptyset$). 
Then we shall observe the relationship between $d(A)$, $d(B)$, $d(\underline{A})$ and $d(\underline{B})$. 
We first notice that $C''+D''$ is a boundary of a compactification of $\bC ^2$. 
By Lemma \ref{lem(2-3-1)}, we have the same formula (\ref{(4-1(1))}). 
Moreover, as the intersection matrix of $D$ is negative definite, we obtain: 
\begin{align}\label{(4-2)}
\left| \begin{array}{ccc}
3 & -1 & -1\\
-d(\underline{A}) & d(A) & 0 \\
-d(\underline{B}) & 0 & d(B) \\
\end{array} \right| {\underset{(\ref{(4-1(1))})}{=}} 2d(A)d(B)-1 >0. 
\end{align}

By the above considerations, we obtain the following lemma: 
\begin{lem}\label{lem(4-2)}
With the notation as above, then the following assertions hold: 
\begin{enumerate}
\item The intersection matrix of $D''$ is negative definite. 
\item If $(C \cdot D^{\sharp})>1$, then $(C'' \cdot {D''}^{\sharp}) \ge 1$. 
\item If $(C \cdot D^{\sharp}) \le 1$, then $(C'' \cdot {D''}^{\sharp})<1$. 
\end{enumerate}
\end{lem}
\begin{proof}
In (1), this assertion is obvious because the weighted dual graph of $D'$ is a subgraph of the weighted dual graph of $D$. 

In (2) and (3), since the intersection matrix of $D''$ is negative definite by (1), there exists uniquely an effective $\bQ$-divisor ${D''}^{\sharp}=\sum _{i''} \alpha _{i''}D_{i''}''$ on $V''$ such that $(D_{i''}'' \cdot K_{V''}+{D''}^{\sharp})=0$ for every irreducible component $D_{i''}''$ of $D''$. 
Let $\alpha _1$ be the coefficient of $D_1$ of $D^{\sharp}$ and let $\alpha _1''$ and $\alpha _2''$ be coefficients of $D_1''$ and $D_2''$ of ${D''}^{\sharp}$, respectively, where we consider $\alpha _2'' := 0$ if $(C'' \cdot D'')=1$. 
When $(C'' \cdot D'')=2$, we may assume that the weighted dual graphs $A$ and $B$ correspond to connected components of $D''$ containing $D_1''$ and $D_2''$, respectively. 
By the Cramer formula (cf.\ (\ref{B0})), we can write: 
\begin{align*}
\alpha _1 = \frac{a}{2d(A)d(B)-1},\ \alpha _1'' + \alpha _2'' = \frac{a_1''}{d(A)} + \frac{a_2''}{d(B)} = \frac{a_1''d(B)+a_2''d(A)}{d(A)d(B)}
\end{align*}
for some positive integers $a$, $a_1''$ and $a_2''$, where $a_2''=0$ if $(C'' \cdot D'')=1$. 
Since $D_1$ is a $(-3)$-curve, we further obtain $a=d(A)d(B) + a_1''d(B)+a_2''d(A)$ by using the cofactor expansion of the determinant. 
Meanwhile, we know that $(C \cdot D^{\sharp}) = \alpha _1$ and $(C'' \cdot {D''}^{\sharp}) = \alpha _1''+\alpha _2''$. 
Thus, we obtain assertions (2) and (3) because $d(A)d(B)$ and $a$ are positive integers. 
\end{proof}
At the end of this subsection, we note that the intersection matrix of $D'$ is not always negative definite. 
Indeed, we can construct the example as follows: 
\begin{eg}
Let $M$ and $F$ be the minimal section and a fiber of the $\bP ^1$-bundle $\bF _2 \to \bP ^1_{\bC}$, respectively, and let $\nu :V \to \bF _2$ be the birational morphism such that $\nu ^{\ast}(M+F)_{\rm red.} = C+D$, where $C$ is a $(-1)$-curve and the weighted dual graph of $C+D$ is as follows: 
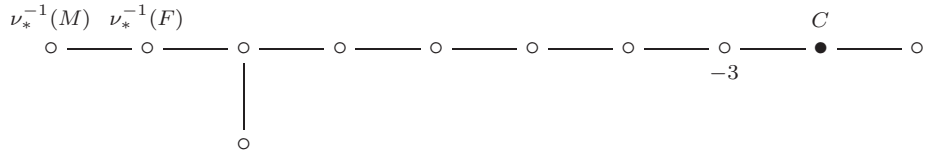
\begin{align*}
\xygraph{
\circ ([]!{+(0,.3)} {^{\nu _{\ast}^{-1}(M)}}) -[r] \circ ([]!{+(0,.3)} {^{\nu _{\ast}^{-1}(F)}}) -[r] \circ (-[d] \circ ,-[r] \circ -[r] \circ -[r] \circ -[r] \circ -[r] \circ ([]!{+(0,-.3)} {^{-3}}) -[r] \bullet ([]!{+(0,.25)} {^C}) -[r] \circ}
\end{align*}
Namely, $(V,C+D)$ is a compactification of $\bC ^2$. 
Notice that an absolute value of the determinant of the intersection matrix of $D$ is equal to $2$. 
Hence, the intersection matrix of $D$ is negative definite; in particular, we obtain the contraction $\pi :V \to X$ of $D$ and a minimal compactification $(X,\pi _{\ast}(C))$ of $\bC ^2$. 

Let $f:V \to V'$ be a contraction of $C$, let $C'$ be the unique $(-1)$-curve on $\Supp (f_{\ast}(D))$ and let us put $D' := f_{\ast}(D)-C'$. 
Then the weighted dual graph of $C'+D'$ is as follows: 
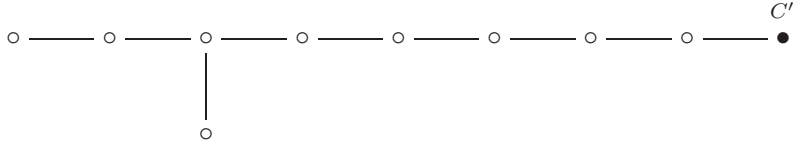
\begin{align*}
\xygraph{
\circ -[r] \circ -[r] \circ (-[d] \circ ,-[r] \circ -[r] \circ -[r] \circ -[r] \circ -[r] \circ -[r] \bullet ([]!{+(0,.25)} {^{C'}})}
\end{align*}
By the configuration of $D'$, the intersection matrix of $D'$ is not negative define. 
\end{eg}
\subsection{Case(3)}\label{4-3}
In this subsection, with the same notation at the beginning of \S \ref{4}, assume further that $D$ and $D'$ are not connected. 
Then $(C \cdot D)=2$ and $(C' \cdot D')=2$. 
More precisely, $C$ meets exactly two irreducible components of $D$, say $D_1$ and $D_2$, $C'$ meets exactly two irreducible components of $D'$, say $D_1'$ and $D_2'$. 
Here, we may assume that $D_2' = f_{\ast}(D_2)$. 
Then we notice $C' = f_{\ast}(D_1)$. 
In particular, $D_1$ is a $(-2)$-curve and $m := -(D_2)^2 \ge 3$. 
Let $A$ and $B$ be weighted dual graphs of two connected components of $D'-C' - D_2'$, where $d(B) := 1$ if $B = \emptyset$. 
We notice $A \not= \emptyset$. 
Then the weighted dual graphs of $C+D$ and $C'+D'$ are as follows: 
\begin{align*}
\xygraph{A -[r] \circ ([]!{+(0,.2)} {^{D_1}}) -[r] \bullet ([]!{+(0,.2)} {^C}) -[r] \circ ([]!{+(0,.2)} {^{D_2}}) ([]!{+(0,-.3)} {^{-m}}) -[r] B}
\quad {\overset{f}{\longrightarrow}} \quad
\xygraph{A -[r] \bullet ([]!{+(0,.2)} {^{C'}}) -[r] \circ ([]!{+(0,.2)} {^{D_2'}}) ([]!{+(0,-.3)} {^{-(m-1)}}) -[r] B}
\end{align*}

Let $D_1'$ and $D_2'$ be irreducible components of $D'$ meeting $C'$, and let $\underline{A}$ and $\overline{B}$ be weighted dual graphs of two connected components of $D'-(D_1'+D_2')$ such that $\underline{A}$ is a subgraph of $A$, where we consider $d(\underline{A}) := 1$ (resp.\ $d(\overline{B}) := 0$, $d(\overline{B}) := 1$) if $\underline{A} = \emptyset$ (resp.\ $B = \emptyset$, $B \not= \emptyset$ and $\overline{B} = \emptyset$). 
Then we shall observe the relationship between $d(A)$, $d(B)$, $d(\underline{A})$ and $d(\overline{B})$. 
We first notice that $C'+D'$ is a boundary of a compactification of $\bC ^2$. 
By Lemma \ref{lem(2-3-1)}, we have: 
\begin{align}\label{(4-3)(1)}
\begin{split}
-1&= \left| \begin{array}{cccc}
d(A) & -d(\underline{A}) & 0 & 0 \\
-1 & 1 & -1 & 0 \\
0 & -1 & m-1 & -1 \\
0 & 0 & -d(\overline{B}) & d(B) \\
\end{array} \right|  \\
&= (m-2) d(A) d(B) - d(A)d(\overline{B}) - (m-1)d(\underline{A})d(B) + d(\underline{A})d(\overline{B}). 
\end{split}
\end{align}
Moreover, as the intersection matrix of $D$ is negative definite, we obtain: 
\begin{align}\label{(4-3)(0)}
2d(A) - d(\underline{A})>0,\ md(B) - d(\overline{B})>0. 
\end{align}

By the above considerations, we obtain the following lemma: 
\begin{lem}\label{lem(4-3)}
With the notation as above, then the following assertions hold: 
\begin{enumerate}
\item The intersection matrix of $D'$ is negative definite. 
\item If $(C \cdot D^{\sharp})>1$, then $(C' \cdot {D'}^{\sharp}) \ge 1$. 
\item If $(C \cdot D^{\sharp})=1$, then $(C' \cdot {D'}^{\sharp}) \le 1$. 
\item If $(C \cdot D^{\sharp})<1$, then $(C' \cdot {D'}^{\sharp})<1$. 
\end{enumerate}
\end{lem}
\begin{proof}
In (1), suppose on the contrary that the intersection matrix of $D'$ is not negative definite. 
Since the weighted dual graph of the connected component of $D'$ with $D_1'$ is a subgraph of the weighted dual graph of $D$, we infer from $d(A) > 0$ that $(m-1)d(B)-d(\overline{B}) \le 0$. 
Then we obtain $d(B) \ge 1$ by virtue of ({\ref{(4-3)(0)}}). 
On the other hand, we know that $A$ is an admissible twig by the assumption and Lemma \ref{lem(3-3)}. 
In particular, $d(A) \ge 1$ and $d(A)-d(\underline{A}) \ge 1$. 
Since the formula ({\ref{(4-3)(1)}}) implies: 
\begin{align*}
-1= \{ d(A)-d(\underline{A})\}\{ (m-1)d(B)-d(\overline{B})\} - d(A)d(B), 
\end{align*}
we thus obtain $(m-1)d(B)-d(\overline{B})=0$ and $d(A)=d(B)=1$. 
Here, $d(A)=1$ implies $A=\emptyset$ because $A$ is an admissible twig. 
However, $D'$ is then connected. 
This is a contradiction to the assumption. 
The assertion (1) is thus proved. 

In (2), since the intersection matrix of $D'$ is negative definite by (1), we see: 
\begin{align}\label{(4-3)(2)}
d(A)\{ (m-1)d(B) - d(\overline{B})\} >0. 
\end{align}
Moreover, there exists uniquely an effective $\bQ$-divisor ${D'}^{\sharp}=\sum _{i'} \alpha _{i'}D_{i'}'$ on $V'$ such that $(D_{i'}' \cdot K_{V'}+{D'}^{\sharp})=0$ for every irreducible component $D_{i'}'$ of $D'$. 
Let $\alpha _1$ and $\alpha _2$ be coefficients of $D_1$ and $D_2$ of $D^{\sharp}$, respectively, and let $\alpha _1'$ and $\alpha _2'$ be coefficients of $D_1'$ and $D_2'$ of ${D'}^{\sharp}$, respectively. 
We may assume that the weighted dual graphs $A$ and $B$ correspond to connected components of $D'$ containing $D_1'$ and $D_2'$, respectively. 
By the Cramer formula (cf.\ (\ref{B0})), we can write: 
\begin{align*}
\alpha _1 = \frac{a_1}{2d(A)-d(\underline{A})},\ \alpha _2 = \frac{a_2}{md(B)-d(\overline{B})},\ 
\alpha _1' = \frac{a_1'}{d(A)},\ \alpha _2' = \frac{a_2'}{(m-1)d(B)-d(\overline{B})}
\end{align*}
for some positive integers $a_1$, $a_2$, $a_1'$ and $a_2'$. 
Since $D_1$ is a $(-2)$-curve, we further obtain $a_1 = a_1'$ by using the cofactor expansion of the determinant. 
Moreover, we know $a_2 = d(B)+a_2'$ by using the cofactor expansion of the determinant again. 
Hence, we see: 
\begin{align}\label{(4-3)(3)}
\alpha _1 = \frac{a_1'}{2d(A)-d(\underline{A})},\ \alpha _2 = \frac{d(B)+a_2'}{md(B)-d(\overline{B})},\ 
\alpha _1' = \frac{a_1'}{d(A)},\ \alpha _2' = \frac{a_2'}{(m-1)d(B)-d(\overline{B})}
\end{align}
Meanwhile, we note that $(C \cdot D^{\sharp}) = \alpha _1+\alpha _2$ and $(C'' \cdot {D''}^{\sharp}) = \alpha _1''+\alpha _2''$. 

Now, in order to show (2), we assume that $(C' \cdot {D'}^{\sharp})<1$. 
In other words, we have: 
\begin{align*}
1-\alpha _1'-\alpha _2' = 1 - \frac{a_1'}{d(A)} -\frac{a_2'}{(m-1)d(B)-d(\overline{B})} > 0. 
\end{align*}
More strictly, we obtain: 
\begin{align}\label{(4-3)(4)}
1-\alpha _1'-\alpha _2' \ge \frac{1}{d(A) \{ (m-1)d(B)-d(\overline{B})\}}
\end{align}
because $d(A) \{ (m-1)d(B)-d(\overline{B})\} (1-\alpha _1'-\alpha _2')$ is a positive integer. 
Hence, we have: 
\begin{align*}
&\{ 2d(A)-d(\underline{A})\} \{ md(B)-d(\overline{B})\} (1-\alpha _1-\alpha _2) \\
&\quad{\underset{(\ref{(4-3)(3)})}{=}} \{ 2d(A)-d(\underline{A})\} \{ md(B)-d(\overline{B})\} \\
&\quad\qquad- \{ md(B)-d(\overline{B})\}a_1' - \{ 2d(A)-d(\underline{A})\} \{ d(B)+a_2'\} \\
&\quad= \{ 2d(A)-d(\underline{A})\} \{ md(B)-d(\overline{B})\} \\
&\quad\qquad- \{ md(B)-d(\overline{B})\} a_1' - \{ 2d(A)-d(\underline{A})\}d(B) - \{ 2d(A)-d(\underline{A})\} a_2' \\
&\quad{\underset{(\ref{(4-3)(3)})}{=}} \{ 2d(A)-d(\underline{A})\} \{ md(B)-d(\overline{B})\} - \{ 2d(A)-d(\underline{A})\}d(B) \\
&\quad\qquad- \{ md(B)-d(\overline{B})\}d(A)\alpha _1' - \{ 2d(A)-d(\underline{A})\} \{ (m-1)d(B)-d(\overline{B})\} \alpha _2' \\
&\quad= d(A)\{ md(B)-d(\overline{B})\}(1-\alpha _1'-\alpha _2') \\
&\quad\qquad+\{(m-2) d(A) d(B) - d(A)d(\overline{B}) - (m-1)d(\underline{A})d(B) + d(\underline{A})d(\overline{B})\} (1-\alpha _2') \\
&\quad{\underset{(\ref{(4-3)(1)}),\,(\ref{(4-3)(4)})}{\ge}} \frac{md(B)-d(\overline{B})}{(m-1)d(B)-d(\overline{B})} -(1-\alpha _2') \\
&\quad= \frac{d(B)}{(m-1)d(B)-d(\overline{B})} +\alpha _2' \\
&\quad>0. 
\end{align*}
This implies that $(C \cdot {D}^{\sharp}) = \alpha _1+ \alpha _2<1$. 

In (3) and (4), with the same notation as in (3), we have: 
\begin{align}\label{(4-3)(5)}
\begin{split}
&d(A) \{ (m-1)d(B)-d(\overline{B})\} (1-\alpha _1'-\alpha _2') \\
&\quad{\underset{(\ref{(4-3)(3)})}{=}} d(A) \{ (m-1)d(B)-d(\overline{B})\} - \{ (m-1)d(B)-d(\overline{B})\} a_1' - d(A)a_2' \\
&\quad= d(A) \{ (m-1)d(B)-d(\overline{B})\} \\
&\quad\qquad- \{ (m-1)d(B)-d(\overline{B})\} a_1' - d(A)\{ d(B)+a_2'\} + d(A)d(B) \\
&\quad{\underset{(\ref{(4-3)(3)})}{=}} d(A) \{ (m-1)d(B)-d(\overline{B})\}  + d(A)d(B)\\
&\quad\qquad- \{ (m-1)d(B)-d(\overline{B})\} \{ 2d(A)-d(\underline{A})\} \alpha _1- d(A) \{ md(B)-d(\overline{B})\} \alpha _2\\
&\quad= d(A) \{ md(B)-d(\overline{B})\} (1-\alpha _1-\alpha _2)\\
&\quad\qquad-\{(m-2) d(A) d(B) - d(A)d(\overline{B}) - (m-1)d(\underline{A})d(B) + d(\underline{A})d(\overline{B})\} \alpha _1 \\
&\quad{\underset{(\ref{(4-3)(1)})}{=}} d(A) \{ md(B)-d(\overline{B})\} (1-\alpha _1-\alpha _2) + \alpha _1\\
&\quad\ge d(A) \{ md(B)-d(\overline{B})\} (1-\alpha _1-\alpha _2). 
\end{split}
\end{align}

Assume that $(C \cdot D^{\sharp})=\alpha _1 + \alpha _2=1$. 
By virtue of (\ref{(4-3)(2)}) and (\ref{(4-3)(5)}), we then obtain $1-\alpha _1'-\alpha _2' \ge 0$. 
This implies that $(C' \cdot {D'}^{\sharp}) \le 1$. 
Assertion (3) is thus proved. 

Assume that $(C \cdot D^{\sharp})=\alpha _1 + \alpha _2<1$. 
By virtue of (\ref{(4-3)(2)}) and (\ref{(4-3)(5)}), we then obtain $1-\alpha _1'-\alpha _2' > 0$. 
This implies that $(C' \cdot {D'}^{\sharp}) < 1$. 
Assertion (4) is thus proved. 
\end{proof}
\section{Proof of Theorem \ref{main}}\label{5}
In this section, we will prove Theorem \ref{main}. 
Notice that this theorem is a consequence of the following: 
\begin{thm}\label{main(2)}
Let $(X,\Gamma )$ be a minimal compactification of the affine plane $\bC ^2$. 
Then we have the following:  
\begin{enumerate}
\item If the canonical divisor $K_X$ is numerically trivial, then $X$ has a singular point, which is not rational. 
\item If the canonical divisor $K_X$ is numerically ample, then $X$ has a singular point, which is not rational. 
\end{enumerate}
\end{thm}
In what follows, we shall prove Theorem \ref{main(2)}. 
We thus prepare some notations. 
Let $(X,\Gamma )$ be a minimal compactification of $\bC ^2$, let $\pi :V \to X$ be the minimal resolution, let $D$ be the reduced exceptional divisor of $\pi$, and $C$ be the proper transform of $\Gamma$ by $\pi$, i.e., $C := \pi _{\ast}^{-1}(\Gamma) $. 
Then $(V,C+D)$ is a compactification of $\bC ^2$. 
Let $D=\sum _iD_i$ be the decomposition of $D$ into irreducible components. 
Since the intersection matrix of $D$ is negative definite, there exists uniquely an effective $\bQ$-divisor $D^{\sharp}=\sum _i \alpha _iD_i$ on $V$ such that $(D_i \cdot K_V+D^{\sharp})=0$ for every irreducible component $D_i$ of $D$. 

From now on, we consider two subsections separately. 
More precisely, we will show Theorem \ref{main(2)} (1) and (2) in \S \S \ref{5-1} and \S \S \ref{5-2}, respectively. 
\subsection{Proof of Theorem \ref{main(2)} (1)}\label{5-1}
With the same notation as above, assume further that $K_X$ is numerically trivial. 
Then $K_V+D^{\sharp} \equiv 0$. 
Noting that $X$ has exactly one singular point $x$ worse than log canonical singularities by {\cite[Theorem 1.1]{KT09}}, we consider the following claim: 
\begin{claim}\label{claim(0)}
$D^{\sharp}$ is a $\bZ$-divisor. 
Moreover, $\Supp (D^{\sharp})$ coincides with the exceptional set of the minimal resolution at $x$. 
\end{claim}
If Claim \ref{claim(0)} is true, then we can prove Theorem \ref{main(2)} (1) as follows: 
\begin{proof}[Proof of ``Claim \ref{claim(0)} $\Rightarrow$ Theorem \ref{main(2)} (1)'']
Let $Z$ be the connected component of $D^{\sharp}$ such that $\Supp (Z)$ coincides with the exceptional set of the minimal resolution at $x$. 
Since we assume that Claim \ref{claim(0)} is true, every coefficient of $Z$ is a positive integer. 
Thus, we have $p_a(Z) = \frac{1}{2}(Z \cdot Z + K_V) +1 = \frac{1}{2}(Z \cdot D^{\sharp} + K_V)+1 = 1>0$. 
By Lemma \ref{rational sing}, this implies that $x \in X$ is an irrational singular point. 
This completes the proof. 
\end{proof}
Hence, we will show Claim \ref{claim(0)} in what follows. 

Since $X$ has exactly one singular point $x$ worse than log canonical singularities, the weighted dual graph $C+D$ is that as in Lemma \ref{lem(3-3)}. 
In other words, we can write $D = \sum _{i=0}^{r+1}D^{(i)}$ $(\exists \, r \in \bZ _{\ge 0})$, $D^{(0)} := \sum _{j=1}^{s_0}D_{1,j}^{(0)} + \sum _{\ell =1}^{t_0}D_{2,\ell}^{(0)}$ and $D^{(i)} := D_0^{(i)} + \sum _{j=1}^{s_i}D_{1,j}^{(i)} + \sum _{\ell =1}^{t_i}D_{2,\ell}^{(i)}$ $(1 \le i \le r+1)$ such that the weighted dual graph $C+D$ is the following: 
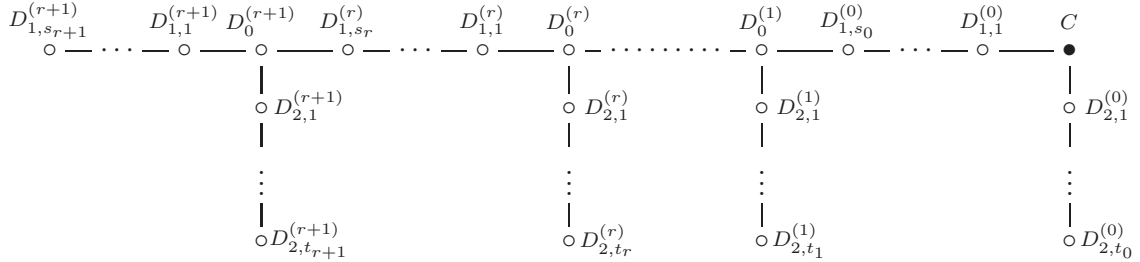
\begin{align*}
\xygraph{
\circ ([]!{+(0,.3)} {^{D_{1,s_{r+1}}^{(r+1)}}}) - []!{+(.7,0)} \cdots - []!{+(.7,0)} \circ ([]!{+(0,.3)} {^{D_{1,1}^{(r+1)}}}) - []!{+(.8,0)} \circ ([]!{+(0,.3)} {^{D_0^{(r+1)}}})
(- []!{+(0,-.6)} \circ ([]!{+(.5,0)} {^{D_{2,1}^{(r+1)}}}) - []!{+(0,-.7)} \vdots - []!{+(0,-.7)} \circ ([]!{+(.5,0)} {^{D_{2,t_{r+1}}^{(r+1)}}}),
- []!{+(.9,0)} \circ ([]!{+(0,.3)} {^{D_{1,s_r}^{(r)}}}) - []!{+(.7,0)} \cdots - []!{+(.7,0)} \circ ([]!{+(0,.3)} {^{D_{1,1}^{(r)}}}) - []!{+(.9,0)} \circ ([]!{+(0,.3)} {^{D_0^{(r)}}})
(- []!{+(0,-.6)} \circ ([]!{+(.4,0)} {^{D_{2,1}^{(r)}}}) - []!{+(0,-.7)} \vdots - []!{+(0,-.7)} \circ ([]!{+(.4,0)} {^{D_{2,t_{r}}^{(r)}}}),
-[r] {\cdots \cdots \cdots} -[r] \circ ([]!{+(0,.3)} {^{D_0^{(1)}}})
(- []!{+(0,-.6)} \circ ([]!{+(.4,0)} {^{D_{2,1}^{(1)}}}) - []!{+(0,-.7)} \vdots - []!{+(0,-.7)} \circ ([]!{+(.4,0)} {^{D_{2,t_{1}}^{(1)}}}),
- []!{+(.9,0)} \circ ([]!{+(0,.3)} {^{D_{1,s_0}^{(0)}}}) - []!{+(.7,0)} \cdots - []!{+(.7,0)} \circ ([]!{+(0,.3)} {^{D_{1,1}^{(0)}}}) - []!{+(.9,0)} \bullet ([]!{+(0,.25)} {^C})
- []!{+(0,-.6)} \circ ([]!{+(.4,0)} {^{D_{2,1}^{(0)}}}) - []!{+(0,-.7)} \vdots - []!{+(0,-.7)} \circ ([]!{+(.4,0)} {^{D_{2,t_0}^{(0)}}})
)))
}
\end{align*}
Here, the above graph satisfies the following conditions: 
\begin{itemize}
\item $s_0,t_0 \ge 0$, where $s_0>0$ if $r=0$. 
\item $s_i \ge 0$, $t_i >0$ for $i=1,\dots ,r$. 
\item $s_{r+1},t_{r+1}>0$. 
\end{itemize}
From now on, for $i=1,\dots ,r$, we set $D_{1,1}^{(i)} := D_0^{(i+1)}$ when $s_i=0$. 
Moreover, for $i=1,\dots ,r+1$, let $\alpha ^{(i)}$ be the coefficient of the irreducible component $D^{(i)}$ of $D^{\sharp}$. 
Similarly, for $i=0,1,\dots ,r+1$, $j = 1,\dots, s_i$ and $\ell = 1,\dots ,t_i$, let $\alpha _{j,\ell}^{(i)}$ be the coefficient of the irreducible component $D_{j,\ell}^{(i)}$ of $D^{\sharp}$, and let $-n_{j,\ell}^{(i)}$ be the self-intersection number of $D_{j,\ell}^{(i)}$. 
\begin{claim}\label{claim(5-1)}
We have $\alpha _0^{(r+1)},\alpha _{1,1}^{(r+1)},\dots ,\alpha _{1,s_{r+1}}^{(r+1)},\alpha _{2,1}^{(r+1)},\dots ,\alpha _{2,t_{r+1}}^{(r+1)} \in \bZ _{>0}$. 
\end{claim}
\begin{proof}
Let $\nu :V \to W$ be the same as in the assumption of Lemma \ref{lem(3-2)}. 
Then we note $W \simeq \bF _m$ for some $m \ge 2$. 
We may assume that $D_{1,s_{r+1}}^{(r+1)}$ is a proper transform of the minimal section on $W \simeq \bF _m$ and $D_{2,t_{r+1}}^{(r+1)}$ is a proper transform of the exceptional curve of the blowing-up at $\nu (C) \in W$. 
In what follows, we use $-K_V \equiv D^{\sharp}$, which is obtained by $K_X \equiv 0$. 

Letting $F$ be a proper transform of a general fiber of the $\bP ^1$-bundle $W \simeq \bF _m \to \bP ^1_{\bC}$ by $\nu$, 
we have $\alpha _{1,s_{r+1}}^{(r+1)}=2$ since $(F \cdot -K_V)=2$ and $(F \cdot D^{\sharp})=(F \cdot \alpha _{1,s_{r+1}}^{(r+1)}D_{1,s_{r+1}}^{(r+1)})=\alpha _{1,s_{r+1}}^{(r+1)}$. 
Moreover, by using Lemma \ref{alpha} (1) and (2) repeatably, we also see that $\alpha _{1,s_{r+1}-1}^{(r+1)}, \dots ,\alpha _{1,1}^{(r+1)},\alpha _0^{(r+1)}$ are positive integers. 

Letting $\Gamma$ be a proper transform of a section, which passes through $\nu (C)$, with self-intersection number $m$ of the $\bP ^1$-bundle $W \simeq \bF _m \to \bP ^1_{\bC}$ by $\nu$, we have $\alpha _{2,t_{r+1}}^{(r+1)}=m+1$ since $(\Gamma \cdot -K_V)=m+1$ and $(\Gamma \cdot D^{\sharp})=(\Gamma \cdot \alpha _{2,t_{r+1}}^{(r+1)}D_{2,t_{r+1}}^{(r+1)})=\alpha _{2,t_{r+1}}^{(r+1)}$. 
Moreover, by using Lemma \ref{alpha} (1) and (2) repeatably, we also see that $\alpha _{2,t_{r+1}-1}^{(r+1)}, \dots ,\alpha _{2,1}^{(r+1)}$ are positive integers. 
\end{proof}
For $i =0,\dots ,r$, let $A_i$ and $B_i$ be the weighted dual graphs defined by $D^{(i+1)}+ \dots + D^{(r)} + \sum _{j=1}^{s_i}D_{1,j}^{(i)}$ and $\sum _{\ell =1}^{t_i}D_{2,\ell}^{(i)}$, respectively. 
Here, $D^{(i+1)}+ \dots + D^{(r)} := 0$ if $i=r$; in other words, $A_r$ is then the weighted dual graph defined by $\sum _{j=1}^{s_r}D_{1,j}^{(r)}$. 
\begin{claim}\label{claim(5-2)}
For $i =0,\dots ,r$, the following assertions hold: 
\begin{enumerate}
\item $d(A_i)d(B_i) \not= 0$. 
\item $d(A_i)$ and $d(B_i)$ are prime to each other. 
In particular, $\left( \frac{1}{d(A_i)}\bZ \right) \cap \left( \frac{1}{d(B_i)}\bZ \right) = \bZ$ as the subset of $\bQ$. 
\item $\alpha _{1,1}^{(0)} \in \frac{1}{d(A_0)}\bZ$ and $\alpha _{1,2}^{(0)} \in \frac{1}{d(B_0)}\bZ$ as subsets of $\bQ$. 
Moreover, for $i=1,\dots ,r$, if $\alpha _0 ^{(i)} \in \bZ$, then, $\alpha _{1,1}^{(i)} \in \frac{1}{d(A_i)}\bZ$ and $\alpha _{1,2}^{(i)} \in \frac{1}{d(B_i)}\bZ$ as subsets of $\bQ$. 
\end{enumerate}
\end{claim}
\begin{proof}
Before the proof, we note the following facts (i) and (ii): 

(i): Assume $i=0$. 
Then $C+D^{(0)}+ \dots +D^{(r)}$ can be contracted to a single $(-1)$-curve passing through the direct image $D_0^{(r+1)}$ such that each $D^{(r+1)}_{1,j}$ and $D^{(r+1)}_{2,\ell}$ have the same self-intersection number of these direct images by the above contraction, respectively. 

(ii): Assume $i>0$. 
Then $C+D^{(0)}+ \dots +D^{(i-1)}$ can be contracted to the divisor corresponding to the following weighted dual graph: 
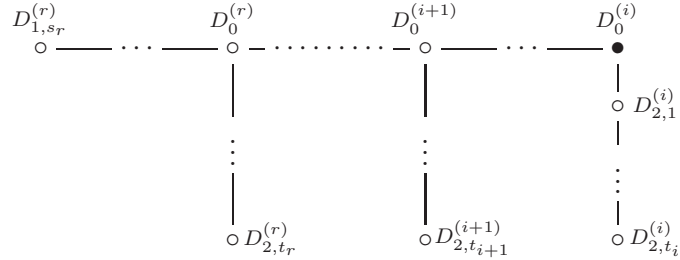
\begin{align*}
\xygraph{
\circ ([]!{+(0,.3)} {^{D_{1,s_r}^{(r)}}}) -[r] \cdots -[r] \circ ([]!{+(0,.3)} {^{D_0^{(r)}}})
(-[d] \vdots -[d] \circ ([]!{+(.4,0)} {^{D_{2,t_r}^{(r)}}}),
-[r] {\cdots \cdots \cdots} -[r] \circ ([]!{+(0,.3)} {^{D_0^{(i+1)}}})
(-[d] \vdots -[d] \circ ([]!{+(.5,0)} {^{D_{2,t_{i+1}}^{(i+1)}}}),
-[r] \cdots -[r] \bullet ([]!{+(0,.3)} {^{D_0^{(i)}}})
- []!{+(0,-.6)} \circ ([]!{+(.4,0)} {^{D_{2,1}^{(i)}}}) - []!{+(0,-.7)} \vdots - []!{+(0,-.7)} \circ ([]!{+(.4,0)} {^{D_{2,t_i}^{(i)}}})
))
}
\end{align*}
Notice that each $D^{(i)}_{1,j}$ and $D^{(i)}_{2,\ell}$ have the same self-intersection number of these direct images by the above contraction, respectively. 
Moreover, the divisor corresponding to the above weighted dual graph can be contracted to a single $(-1)$-curve passing through the direct image $D_0^{(r+1)}$. 

In (1), we note that $\sum _{\ell =1}^{t_i}D_{2,\ell}^{(i)}$ is a rational chain with self-intersection number $\le -2$. 
Hence, we obtain $d(B_i) >1$ by the induction on $t_i$. 
Now, suppose on the contrary $d(A_i)=0$. 
Then we obtain $-a \cdot d(B_i)=1$ for some $a \in \bZ$ by facts (i) and (ii) at the beginning (see also (\ref{det})). 
In particular, we have $d(B_i)= \pm 1$ because $a$ is an integer; however, it contradicts $d(B_i)>1$. 
Thus, $d(A_i) \not= 0$. 

In (2), it follows from Lemma \ref{lem(2-3-2)} combined with facts (i) and (ii) at the beginning. 

In (3), we shall first show $\alpha _{2,1}^{(0)} \in \frac{1}{d(B_0)}\bZ$. 
By virtue of the equations $\{ (D_{2,\ell}^{(0)} \cdot D^{\sharp})=(D_{2,\ell}^{(0)} \cdot -K_V)\} _{1 \le \ell \le t_0}$, we obtain a linear simultaneous equation, whose left-hand-side corresponds to the intersection matrix of $B_0$. 
In other words, we can write: 
\begin{align}\label{B0}
\left[ \begin{array}{ccccc}
n_{2,1}^{(0)} & -1 & & & \\
-1 & n_{2,2}^{(0)} & -1 & & \\
 & -1 & \ddots & \ddots & \\
 & & \ddots & \ddots & -1 \\
 & &  & -1 & n_{2,t_0}^{(0)} 
\end{array} \right]
\left[ \begin{array}{c}
\alpha _{2,1}^{(0)} \\ \alpha _{2,2}^{(0)} \\ \vdots \\ \vdots \\ \alpha _{2,t_0}^{(0)}
\end{array} \right]
=
\left[ \begin{array}{c}
-(D_{2,1}^{(0)} \cdot K_V) \\ -(D_{2,2}^{(0)} \cdot K_V) \\ \vdots \\ \vdots \\ -(D_{2,t_0}^{(0)} \cdot K_V)
\end{array} \right]
\end{align}
Note $d(B_0) \not= 0$ by (1). 
Hence, we know that $\alpha _{2,1}^{(0)} \in \frac{1}{d(B_0)}\bZ$ by the Cramer formula. 

Next, we shall show $\alpha _{1,1}^{(0)} \in \frac{1}{d(A_0)}\bZ$. 
As in the proof of $\alpha _{2,1}^{(0)} \in \frac{1}{d(B_0)}\bZ$, we consider the equations $\{ (D_{1,j_i}^{(i)} \cdot D^{\sharp})=(D_{1,j_i}^{(i)} \cdot -K_V)\} _{0 \le i \le r,\, 1 \le j_i \le s_i}$ and $\{ (D_{2,\ell _i}^{(i)} \cdot D^{\sharp})=(D_{2,\ell _i}^{(i)} \cdot -K_V)\} _{1 \le i \le r,\, 1 \le \ell _i \le t_i}$. 
Then we obtain a linear simultaneous equation similar to (\ref{B0}), whose left-hand-side corresponds to the intersection matrix of $A_0$. 
Here, we note that any component of the right-hand-side of this linear simultaneous equation is an integer because of: 
\begin{align*}
&(D_{1,s_r}^{(r)} \cdot D_{1,s_r-1}^{(r)})\alpha _{1,s_r-1}^{(r)} + (D_{1,s_r}^{(r)} \cdot D_{1,s_r}^{(r)})\alpha _{1,s_r}^{(r)} + (D_{1,s_r}^{(r)} \cdot D_0^{(r+1)})\alpha _0^{(r+1)} = (D_{1,s_r}^{(r)} \cdot -K_V) \\
\iff &(D_{1,s_r}^{(r)} \cdot D_{1,s_r-1}^{(r)})\alpha _{1,s_r-1}^{(r)} + (D_{1,s_r}^{(r)} \cdot D_{1,s_r}^{(r)})\alpha _{1,s_r}^{(r)} = 
(D_{1,s_r}^{(r)} \cdot -K_V) - (D_{1,s_r}^{(r)} \cdot D_0^{(r+1)})\alpha _0^{(r+1)}
\end{align*}
and $\alpha _0^{(r+1)} \in \bZ$ by Claim \ref{claim(5-1)}. 
Hence, we know that $\alpha _{1,1}^{(0)} \in \frac{1}{d(A_0)}\bZ$ by the Cramer formula. 

The remaining follows from a similar argument as above. 
\end{proof}
By using Claim \ref{claim(5-2)}, we shall prove Claims \ref{claim(5-3)} and \ref{claim(5-4)}: 
\begin{claim}\label{claim(5-3)}
We have $\alpha _{1,1}^{(0)},\dots ,\alpha _{1,s_0}^{(0)},\alpha _0^{(1)} \in \bZ _{>0}$. 
Moreover, if $t_0>0$, then we have $\alpha _{2,\ell}^{(0)}=0$ for every $\ell = 1,\dots ,t_0$. 
\end{claim}
\begin{proof}
At first, we prove $\alpha _{2,1}^{(0)}=0$ (if $t_0>0$) and $\alpha _{1,1}^{(0)}=1$. 
Notice $1=(C \cdot -K_V) = (C \cdot D^{\sharp})$ because $C$ is a $(-1)$-curve and $K_V+D^{\sharp} \equiv 0$. 
Hence, if $t_0=0$, then we obtain $1=(C \cdot D^{\sharp}) = \alpha _{1,1}^{(0)}$. 
From now on, we assume $t_0>0$. 
Then $\alpha _{1,1}^{(0)}+\alpha _{2,1}^{(0)}=1$. 
Meanwhile, since $\alpha _{1,1}^{(0)} \in \frac{1}{d(A_0)}\bZ$ and $\alpha _{2,1}^{(0)} \in \frac{1}{d(B_0)}\bZ$ as subsets of $\bQ$ by Claim \ref{claim(5-2)} (3), 
we have $\alpha _{2,1}^{(0)} \in \left( \frac{1}{d(A_0)}\bZ \right) \cap \left( \frac{1}{d(B_0)}\bZ \right) = \bZ$ by using Claim \ref{claim(5-2)} (2). 
Moreover, we know $\alpha _{2,1}^{(0)} \in \{ 0,1\}$ because $\alpha _{1,1}^{(0)}$ and $\alpha _{2,1}^{(0)}$ are non-negative. 
Here, we note $\alpha _{2,1}^{(0)}=0$. 
Indeed, $\sum _{\ell =1}^{t_0}D_{2,\ell}^{(0)}$ is a connected component of $D$ and can be contracted to a cyclic quotient singular point. 
Hence, we obtain $\alpha _{2,1}^{(0)}=0$. 
Namely, $\alpha _{1,1}^{(0)}=1$. 

In what follows, we shall show $\alpha _{2,2}^{(0)} = \dots = \alpha _{2,t_0}^{(0)} = 0$ (if $t_0>0$) and $\alpha _{1,2}^{(0)},\dots ,\alpha _{1,{s_0}}^{(0)},\alpha _0^{(1)} \in \bZ _{>0}$. 
However, in this case, we obtain this assertion by using Lemma \ref{alpha} repeatedly. 
\end{proof}
\begin{claim}\label{claim(5-4)}
For $i =1,\dots ,r$, we have $\alpha _{1,1}^{(i)},\dots ,\alpha _{1,s_i}^{(i)},\alpha _{2,1}^{(i)},\dots ,\alpha _{2,t_i}^{(i)},\alpha _0^{(i+1)} \in \bZ _{>0}$. 
\end{claim}
\begin{proof}
At first, we shall treat the case $i=1$. 
By virtue of $(D_0^{(1)} \cdot -K_V) = (D_0^{(1)} \cdot D^{\sharp})$ combined with Claim \ref{claim(5-3)}, we have $\alpha _0 ^{(1)} \in \bZ _{>0}$ and $\alpha _{1,1}^{(1)}+\alpha _{2,1}^{(1)} \in \bZ$. 
Meanwhile, since $\alpha _{1,1}^{(1)} \in \frac{1}{d(A_1)}\bZ$ and $\alpha _{2,1}^{(1)} \in \frac{1}{d(B_1)}\bZ$ as subsets of $\bQ$ by Claim \ref{claim(5-2)} (3), we have $\alpha _{1,1}^{(1)}, \alpha _{2,1}^{(1)} \in \left( \frac{1}{d(A_1)}\bZ \right) \cap \left( \frac{1}{d(B_1)}\bZ \right) = \bZ$ by Claim \ref{claim(5-2)} (2). 
Here, we note that $\alpha _{1,1}^{(1)}$ and $\alpha _{2,1}^{(1)}$ are positive integers by Lemma \ref{alpha} (1). 
Hence, we know that $\alpha _{1,j}^{(1)}$, $\alpha _{2,\ell}^{(1)}$ and $\alpha _0^{(2)}$ are positive integers by using Lemma \ref{alpha} repeatedly. 

The remaining follows from the similar argument as above. 
\end{proof}
Claim \ref{claim(0)} follows from Claims \ref{claim(5-1)}, \ref{claim(5-3)} and \ref{claim(5-4)}. 
The proof of Theorem \ref{main(2)} (1) is thus completed. 

At the end of this subsection, we provide some examples of minimal compactifications of $\bC ^2$ with the numerically trivial canonical divisors: 
\begin{eg}\label{eg(5-1)(1)}
We can construct a compactification $(V,C+D)$ of $\bC ^2$ such that $V$ is a smooth projective surface and the weighted dual graph of $C+D$ is as follows: 
\begin{align*}
\xygraph{
\circ ([]!{+(0,-.3)} {^{-m}}) ([]!{+(0,.3)} {^{D_{1,2}^{(1)}}}) -[r] \circ ([]!{+(0,.3)} {^{D_{1,1}^{(1)}}}) -[r] \circ ([]!{+(0,.3)} {^{D_0^{(1)}}}) (-[d] \circ ([]!{+(.4,0)} {^{D_{2,1}^{(1)}}}),-[r] \circ ([]!{+(0,.3)} {^{D_{1,2m+1}^{(0)}}}) -[r] {\cdots} ([]!{+(0,-.4)} {\underbrace{\quad \qquad \qquad \qquad}_{(2m+1)\text{-vertices}}}) -[r] \circ ([]!{+(0,.3)} {^{D_{1,1}^{(0)}}}) -[r] \bullet ([]!{+(0,.25)} {^C})))} \quad (m \ge 3)
\end{align*}
Indeed, there exists a birational morphism $\nu :V \to \bF _m$ such that $(\bF _m,\nu _{\ast}(C+D))$ is a minimal normal compactification of $\bC ^2$ as in Lemma \ref{lem(3-2)}. 
By a straightforward calculation, we know that the determinant of the intersection matrix of $D$ is equal to $2(m-2)$, so that the intersection matrix of $D$ is negative definite. 
Hence, by Lemma \ref{cont}, there exists a contraction $\pi :V \to X$ of $D$ such that $(X,\pi (C))$ is a minimal compactification of $\bC ^2$. 
Let $D^{\sharp}$ be the $\bQ$-divisor on $V$ such that $K_V + D^{\sharp} \equiv \pi ^{\ast}(K_X)$. 
Then coefficients of $D^{\sharp}$ are determined as follows: 
\begin{align*}
\alpha _{1,i}^{(0)} = i\ (i=1,\dots ,2m+1),\ \alpha _0^{(1)} = 2(m+1),\ \alpha _{1,1}^{(1)} = m+2,\ \alpha _{1,2}^{(1)} = 2,\ \alpha _{2,1}^{(1)} = m+1. 
\end{align*}
Hence, by Lemma \ref{numerical}, $K_X$ is numerically trivial. 
In particular, $\pi (D) \in X$ is an irrational singular point by Theorem \ref{main(2)} (1). 
\end{eg}
\begin{eg}\label{eg(5-1)(2)}
We can construct a compactification $(V,C+D)$ of $\bC ^2$ such that $V$ is a smooth projective surface and the weighted dual graph of $C+D$ is as follows: 
\begin{align*}
\xygraph{
\circ ([]!{+(0,.3)} {^{D_{1,2}^{(1)}}}) -[r] \circ ([]!{+(0,.3)} {^{D_{1,1}^{(1)}}}) -[r] \circ ([]!{+(0,.3)} {^{D_0^{(1)}}}) (-[d] \circ ([]!{+(.3,0)} {^{D_{2,1}^{(1)}}}),-[r] \circ ([]!{+(0,.3)} {^{D_{1,5}^{(0)}}}) -[r] {\cdots} ([]!{+(0,-.4)} {\underbrace{\quad \qquad \qquad \qquad}_{4\text{-vertices}}}) -[r] \circ ([]!{+(0,.3)} {^{D_{1,2}^{(0)}}}) -[r] \circ ([]!{+(0,-.3)} {^{-m}}) ([]!{+(0,.3)} {^{D_{1,1}^{(0)}}}) -[r] \bullet ([]!{+(0,.25)} {^C}) -[r] \circ ([]!{+(0,.3)} {^{D_{2,1}^{(0)}}}) -[r] {\cdots} ([]!{+(0,-.4)} {\underbrace{\quad \qquad \qquad \qquad}_{(m-2)\text{-vertices}}}) -[r] \circ ([]!{+(0,.3)} {^{D_{2,m-2}^{(0)}}}))} \quad (m \ge 3)
\end{align*}
Indeed, there exists a birational morphism $\nu :V \to \bF _2$ such that $(\bF _2,\nu _{\ast}(C+D))$ is a minimal normal compactification of $\bC ^2$ as in Lemma \ref{lem(3-2)}. 
By a straightforward calculation, we know that the determinant of the intersection matrix of $D$ is equal to $(m-2)(m-1)$, so that the intersection matrix of $D$ is negative definite. 
Hence, by Lemma \ref{cont}, there exists a contraction $\pi :V \to X$ of $D$ such that $(X,\pi (C))$ is a minimal compactification of $\bC ^2$. 
Let $D^{\sharp}$ be the $\bQ$-divisor on $V$ such that $K_V + D^{\sharp} \equiv \pi ^{\ast}(K_X)$. 
Then coefficients of $D^{\sharp}$ are determined as follows: 
\begin{align*}
\alpha _{1,i}^{(0)} = i\ (i=1,2,3,4,5),\ &\alpha _{2,j}^{(0)} = 0\ (j=1,\dots ,m-2),\\
\alpha _0^{(1)} = 6,\ \alpha _{1,1}^{(1)} = 4,\ &\alpha _{1,2}^{(1)} = 2,\ \ \alpha _{2,1}^{(1)} = 3. 
\end{align*}
Hence, by Lemma \ref{numerical}, $K_X$ is numerically trivial. 
In particular, $\Sing (X)$ consists of a cyclic quotient singular point and an irrational singular point by Theorem \ref{main(2)} (1). 
\end{eg}
\subsection{Proof of Theorem \ref{main(2)} (2)}\label{5-2}
Let the notation be the same at the beginning \S \ref{5}, and assume further that the canonical divisor $K_X$ is numerically ample. 
By Lemmas \ref{numerical}, \ref{lem(4-1)}, \ref{lem(4-2)} and \ref{lem(4-3)}, we then obtain the following lemma: 
\begin{lem}\label{lem(5-2)}
Let the notation be the same at the beginning \S \ref{5}, and assume further that $K_X$ is numerically ample. 
Then there exists a birational morphism $\widetilde{f}:V \to \widetilde{V}$ and there exists a unique $(-1)$-curve $\widetilde{C}$ on $\widetilde{f}_{\ast}(D)$ such that the following properties on the divisor $\widetilde{D} := \widetilde{f}_{\ast}(D)-\widetilde{C}$ hold: 
\begin{itemize}
\item $V \backslash \Supp (C+D) \simeq \widetilde{V} \backslash \Supp (\widetilde{C}+\widetilde{D})$. 
\item The intersection matrix of $\widetilde{D}$ is negative definite. 
\item Let $\widetilde{\pi}:\widetilde{V} \to \widetilde{X}$ be the contraction of $\widetilde{D}$ and let us put $\widetilde{\Gamma} := \widetilde{\pi}_{\ast}(\widetilde{C})$. 
Then $(\widetilde{X},\widetilde{\Gamma})$ be a minimal compactification of $\bC^2$ such that $K_{\widetilde{X}}$ is numerically trivial. 
\end{itemize}
\end{lem}
By using Lemma \ref{lem(5-2)}, we can show Theorem \ref{main(2)} (2) as follows: 
\begin{proof}[Proof of Theorem \ref{main(2)} (2)]
Let $\widetilde{D}$ and $(\widetilde{X},\widetilde{\Gamma})$ be the same as in Lemma \ref{lem(5-2)}. 
Let $\widetilde{D}=\sum _j\widetilde{D}_j$ be the decomposition of $\widetilde{D}$ into irreducible components. 
Since the intersection matrix of $\widetilde{D}$ is negative definite, there exists uniquely an effective $\bQ$-divisor $\widetilde{D}^{\sharp}=\sum _j\widetilde{\alpha}_j\widetilde{D}_j$ on $\widetilde{V}$ such that $(\widetilde{D}_j \cdot K_{\widetilde{V}}+\widetilde{D}^{\sharp})=0$ for every irreducible component $\widetilde{D}_j$ of $\widetilde{D}$. 
Moreover, letting $\widetilde{Z}$ be the connected component of $\widetilde{D}^{\sharp}$ corresponding to a singular point worse than log canonical singularities, any coefficient of $\widetilde{Z}$ is a positive integer by Claim \ref{claim(0)}. 
Now, letting $E$ be the connected component of $D$ corresponding to a singular point worse than log canonical singularities, we shall put:
\begin{align}\label{Z}
Z := \widetilde{f}^{-1}_{\ast}(\widetilde{Z}) + (E - \widetilde{f}^{-1}_{\ast}(\widetilde{Z})_{\rm red.}). 
\end{align}
Then we have: 
\begin{align*}
(Z \cdot K_V+Z) &= (\widetilde{f}^{-1}_{\ast}(\widetilde{Z}) \cdot K_V + \widetilde{f}^{-1}_{\ast}(\widetilde{Z})) + 2 (\widetilde{f}^{-1}_{\ast}(\widetilde{Z}) \cdot E - \widetilde{f}^{-1}_{\ast}(\widetilde{Z})_{\rm red.})\\
&\qquad+ (E - \widetilde{f}^{-1}_{\ast}(\widetilde{Z})_{\rm red.} \cdot K_V + E - \widetilde{f}^{-1}_{\ast}(\widetilde{Z})_{\rm red.}). 
\end{align*}
Here, by construction of $\widetilde{f}$ and $\widetilde{Z}$ we notice: 
\begin{align*}
(\widetilde{f}^{-1}_{\ast}(\widetilde{Z}) \cdot K_V + \widetilde{f}^{-1}_{\ast}(\widetilde{Z})) = (\widetilde{Z} \cdot K_{\widetilde{V}} + \widetilde{Z}) = 0, \\
(\widetilde{f}^{-1}_{\ast}(\widetilde{Z}) \cdot E - \widetilde{f}^{-1}_{\ast}(\widetilde{Z})_{\rm red.}) = (\widetilde{Z} \cdot \widetilde{C})=1.
\end{align*}
Moreover, since $E - \widetilde{f}^{-1}_{\ast}(\widetilde{Z})_{\rm red.}$ is reduced and its dual graph is a tree, we know: 
\begin{align*}
&(E - \widetilde{f}^{-1}_{\ast}(\widetilde{Z})_{\rm red.} \cdot K_V + E - \widetilde{f}^{-1}_{\ast}(\widetilde{Z})_{\rm red.}) \\
&\quad= 
\sharp \{ D_i\,|\, D_i\text{ is a branching component of }E - \widetilde{f}^{-1}_{\ast}(\widetilde{Z})_{\rm red.}\} \\
&\qquad -\sharp \{ D_i\,|\, D_i\text{ is a terminal component of }E - \widetilde{f}^{-1}_{\ast}(\widetilde{Z})_{\rm red.}\} \\
&\quad=-2. 
\end{align*}
In summary, we have $p_a(Z) = \frac{1}{2}(Z \cdot K_V+Z)+1 = 1>0$. 
This implies that $\pi (E) \in X$ is a singular point, which is not rational, by Lemma \ref{rational sing}. 
\end{proof}
In the above proof, we present some examples for the reader's convenience: 
\begin{eg}\label{eg(5-2)(1)}
We can construct a compactification $(V,C+D)$ of $\bC ^2$ such that $V$ is a smooth projective surface and the weighted dual graph of $C+D$ is as follows: 
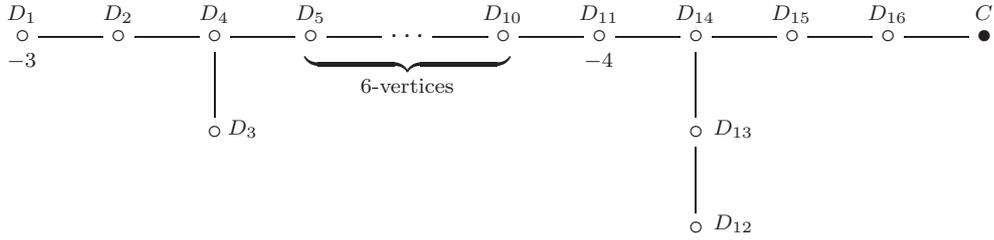
\begin{align*}
\xygraph{
\circ ([]!{+(0,-.3)} {^{-3}}) ([]!{+(0,.2)} {^{D_1}}) -[r] \circ ([]!{+(0,.2)} {^{D_2}}) -[r] \circ ([]!{+(0,.2)} {^{D_4}}) (-[d] \circ ([]!{+(.3,0)} {^{D_3}}),-[r] \circ ([]!{+(0,.2)} {^{D_5}}) -[r] {\cdots} ([]!{+(0,-.4)} {\underbrace{\quad \qquad \qquad \qquad}_{6\text{-vertices}}}) -[r] \circ ([]!{+(0,.2)} {^{D_{10}}}) -[r] \circ ([]!{+(0,-.3)} {^{-4}}) ([]!{+(0,.2)} {^{D_{11}}}) -[r] \circ ([]!{+(0,.2)} {^{D_{14}}}) (-[d] \circ ([]!{+(.4,0)} {^{D_{13}}}) -[d] \circ ([]!{+(.4,0)} {^{D_{12}}}),-[r] \circ ([]!{+(0,.2)} {^{D_{15}}}) -[r] \circ ([]!{+(0,.2)} {^{D_{16}}}) -[r] \bullet ([]!{+(0,.2)} {^C}))))}
\end{align*}
Here, in the above graph, each $D_i$ is an irreducible component of $D$. 
Indeed, there exists a birational morphism $\nu :V \to \bF _3$ such that $(\bF _3,\nu _{\ast}(C+D))$ is a minimal normal compactification of $\bC ^2$ as in Lemma \ref{lem(3-2)}. 
By a straightforward calculation, we know that the determinant of the intersection matrix of $D$ is equal to $21$. 
Hence, there exists a contraction $\pi :V \to X$ such that $(X,\Gamma )$ is a minimal compactification of $\bC ^2$ by Lemma \ref{cont}, where $\Gamma := \pi (C)$. 
Let $D^{\sharp} = \sum _{i=1}^{16}\alpha _iD_i$ be the $\bQ$-divisor on $V$ such that $K_V + D^{\sharp} \equiv \pi ^{\ast}(K_X)$. 
Then coefficients of $D^{\sharp}$ are determined as follows: 
\begin{align*}
\alpha _1 = \frac{20}{7},\ 
&\alpha _2 = \frac{53}{7},\ 
\alpha _3 = \frac{43}{7},\ 
\alpha _i = \frac{126-10i}{7}\ (i=4,\dots ,11),\ \\
&\alpha _{12} = \alpha _{16} = \frac{8}{7},\ \alpha _{13} = \alpha _{15} = \frac{16}{7},\ \alpha _{14} = \frac{24}{7}. 
\end{align*}
Hence, $(X,\Gamma )$ satisfies the assumption of \S \S \ref{5-2} by Lemma \ref{numerical} because of $(C \cdot D^{\sharp}) = \alpha _{16} = \frac{8}{7}>1$. 
Let $\widetilde{f}:V \to \widetilde{V}$ be a sequence of contractions of $(-1)$-curves and subsequently (smoothly) contractible curves in $\Supp (D)$, starting with the contraction of $C$, such that $V \backslash \Supp (C+D) \simeq \widetilde{V} \backslash \Supp (\widetilde{f}_{\ast}(D))$ and the weighted dual graph of $\widetilde{f}_{\ast}(D)$ is as follows: 
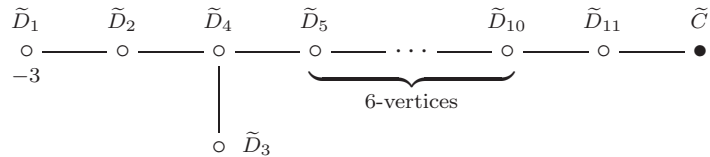
\begin{align*}
\xygraph{
\circ ([]!{+(0,-.3)} {^{-3}}) ([]!{+(0,.3)} {^{\widetilde{D}_1}}) -[r] \circ ([]!{+(0,.3)} {^{\widetilde{D}_2}}) -[r] \circ ([]!{+(0,.3)} {^{\widetilde{D}_4}}) (-[d] \circ ([]!{+(.4,0)} {^{\widetilde{D}_3}}),-[r] \circ ([]!{+(0,.3)} {^{\widetilde{D}_5}}) -[r] {\cdots} ([]!{+(0,-.4)} {\underbrace{\quad \qquad \qquad \qquad}_{6\text{-vertices}}}) -[r] \circ ([]!{+(0,.3)} {^{\widetilde{D}_{10}}}) -[r] \circ ([]!{+(0,.3)} {^{\widetilde{D}_{11}}}) -[r] \bullet([]!{+(0,.3)} {^{\widetilde{C}}})))}
\end{align*}
Here, in the above graph, $\widetilde{C}$ is a unique $(-1)$-curve on $\Supp (\widetilde{f}_{\ast}(D))$ and each $\widetilde{D}_j$ is an irreducible component of $\widetilde{D} := \widetilde{f}_{\ast}(D) - \widetilde{C}$. 
Since $(\widetilde{V},\widetilde{C}+\widetilde{D})$ is the same as $(V,C+D)$ in Example \ref{eg(5-1)(1)} for $m=3$, we know that $(\widetilde{V},\widetilde{C}+\widetilde{D})$ satisfies the assumption for $(V,C+D)$ in \S \S \ref{5-1}. 
In particular, there exists a contraction $\widetilde{\pi}:\widetilde{V} \to \widetilde{X}$ of $\widetilde{D}$ such that $K_{\widetilde{X}} \equiv 0$. 
Let $\widetilde{D}^{\sharp} = \sum _{j=1}^{11}{\widetilde{\alpha}}_j\widetilde{D}_j$ be the $\bQ$-divisor on $\widetilde{V}$ such that $K_{\widetilde{V}} + \widetilde{D}^{\sharp} \equiv \widetilde{\pi}^{\ast}(K_{\widetilde{X}})$. 
By Example \ref{eg(5-1)(1)}, we then have $\widetilde{\alpha}_1 = 2$, $\widetilde{\alpha}_2 = 5$, $\widetilde{\alpha}_3 = 4$ and $\widetilde{\alpha}_j = 12-j$ $(j=4,\dots ,11)$. 
Now, we shall set the $\bZ$-divisor $Z = \sum _{i=1}^{16}x_iD_i$ on $V$ as ({\ref{Z}}). 
Noting that $D_j = \widetilde{f}_{\ast}^{-1}(\widetilde{D} _j )$ for $j=1,\dots ,11$, we have: 
\begin{align*}
x_1 = 2,\ x_2 = 5,\ x_3 = 4,\ x_j = 12-j\ (j=4,\dots ,11),\ x_{12} = \dots =x_{16} = 1. 
\end{align*}
Then we see that $\Supp (Z) = \Supp (D)$ and $p_a(Z) = 1$ by a straightforward calculation. 
\end{eg}
\begin{eg}\label{eg(5-2)(2)}
We can construct a compactification $(V,C+D)$ of $\bC ^2$ such that $V$ is a smooth projective surface and the weighted dual graph of $C+D$ is as follows: 
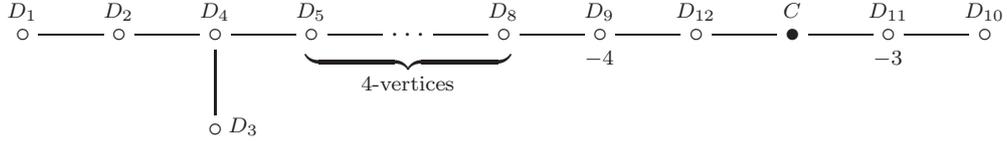
\begin{align*}
\xygraph{
\circ ([]!{+(0,.2)} {^{D_1}}) -[r] \circ ([]!{+(0,.2)} {^{D_2}}) -[r] \circ ([]!{+(0,.2)} {^{D_4}}) (-[d] \circ ([]!{+(.3,0)} {^{D_3}}),-[r] \circ ([]!{+(0,.2)} {^{D_5}}) -[r] {\cdots} ([]!{+(0,-.4)} {\underbrace{\quad \qquad \qquad \qquad}_{4\text{-vertices}}}) -[r] \circ ([]!{+(0,.2)} {^{D_8}}) -[r] \circ ([]!{+(0,-.3)} {^{-4}}) ([]!{+(0,.2)} {^{D_9}}) -[r] \circ ([]!{+(0,.2)} {^{D_{12}}}) -[r] \bullet ([]!{+(0,.2)} {^C}) -[r] \circ ([]!{+(0,-.3)} {^{-3}}) ([]!{+(0,.2)} {^{D_{11}}}) -[r] \circ ([]!{+(0,.2)} {^{D_{10}}}))}
\end{align*}
Here, in the above graph, each $D_i$ is an irreducible component of $D$. 
Indeed, there exists a birational morphism $\nu :V \to \bF _2$ such that $(\bF _2,\nu _{\ast}(C+D))$ is a minimal normal compactification of $\bC ^2$ as in Lemma \ref{lem(3-2)}. 
By a straightforward calculation, we know that the determinant of the intersection matrix of $D$ is equal to $3 \cdot 5 = 15$. 
Hence, there exists a contraction $\pi :V \to X$ such that $(X,\Gamma )$ is a minimal compactification of $\bC ^2$ by Lemma \ref{cont}, where $\Gamma := \pi (C)$. 
Let $D^{\sharp} = \sum _{i=1}^{12}\alpha _iD_i$ be the $\bQ$-divisor on $V$ such that $K_V + D^{\sharp} \equiv \pi ^{\ast}(K_X)$. 
Then coefficients of $D^{\sharp}$ are determined as follows: 
\begin{align*}
\alpha _1 = \frac{8}{3},\ 
\alpha _2 = \frac{16}{3},\ 
\alpha _3 = 4,\ 
\alpha _i = \frac{40-4i}{3}\ (i=4,\dots ,9),\ 
\alpha _{10} = \frac{1}{5},\ 
\alpha _{11} = \frac{2}{5},\ 
\alpha _{12} = \frac{2}{3}. 
\end{align*}
Hence, $(X,\Gamma )$ satisfies the assumption of \S \S \ref{5-2} by Lemma \ref{numerical} because of $(C \cdot D^{\sharp}) = \alpha _{12} + \alpha _{11} = \frac{2}{3} + \frac{2}{5} = \frac{16}{15}>1$. 
Let $\widetilde{f}:V \to \widetilde{V}$ be a sequence of contractions of $(-1)$-curves and subsequently (smoothly) contractible curves in $\Supp (D)$, starting with the contraction of $C$, such that $V \backslash \Supp (C+D) \simeq \widetilde{V} \backslash \Supp (\widetilde{f}_{\ast}(D))$ and the weighted dual graph of $\widetilde{f}_{\ast}(D)$ is as follows: 
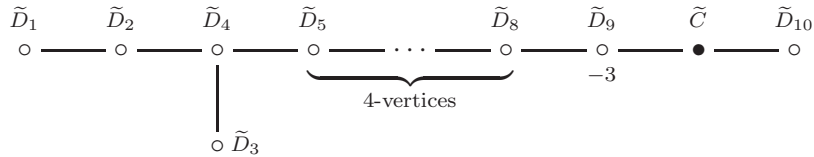
\begin{align*}
\xygraph{
\circ ([]!{+(0,.3)} {^{\widetilde{D}_1}}) -[r] \circ ([]!{+(0,.3)} {^{\widetilde{D}_2}}) -[r] \circ ([]!{+(0,.3)} {^{\widetilde{D}_4}}) (-[d] \circ ([]!{+(.3,0)} {^{\widetilde{D}_3}}),-[r] \circ ([]!{+(0,.3)} {^{\widetilde{D}_5}}) -[r] {\cdots} ([]!{+(0,-.4)} {\underbrace{\quad \qquad \qquad \qquad}_{4\text{-vertices}}}) -[r] \circ ([]!{+(0,.3)} {^{\widetilde{D}_8}}) -[r] \circ ([]!{+(0,-.3)} {^{-3}}) ([]!{+(0,.3)} {^{\widetilde{D}_9}}) -[r] \bullet ([]!{+(0,.3)} {^{\widetilde{C}}}) -[r] \circ ([]!{+(0,.3)} {^{\widetilde{D}_{10}}}))}
\end{align*}
Here, in the above graph, $\widetilde{C}$ is a unique $(-1)$-curve on $\Supp (\widetilde{f}_{\ast}(D))$ and each $\widetilde{D}_j$ is an irreducible component of $\widetilde{D} := \widetilde{f}_{\ast}(D) - \widetilde{C}$. 
Since $(\widetilde{V},\widetilde{C}+\widetilde{D})$ is the same as $(V,C+D)$ in Example \ref{eg(5-1)(2)} for $m=3$, we know that $(\widetilde{V},\widetilde{C}+\widetilde{D})$ satisfies the assumption for $(V,C+D)$ in \S \S \ref{5-1}. 
In particular, there exists a contraction $\widetilde{\pi}:\widetilde{V} \to \widetilde{X}$ of $\widetilde{D}$ such that $K_{\widetilde{X}} \equiv 0$. 
Let $\widetilde{D}^{\sharp} = \sum _{j=1}^{10}{\widetilde{\alpha}}_j\widetilde{D}_j$ be the $\bQ$-divisor on $\widetilde{V}$ such that $K_{\widetilde{V}} + \widetilde{D}^{\sharp} \equiv \widetilde{\pi}^{\ast}(K_{\widetilde{X}})$. 
By Example \ref{eg(5-1)(2)}, we then have $\widetilde{\alpha}_1 = 2$, $\widetilde{\alpha}_2 = 4$, $\widetilde{\alpha}_3 = 3$ and $\widetilde{\alpha}_j = 10-j$ $(j=4,\dots ,9)$. 
Now, we shall set the $\bZ$-divisor $Z = \sum _{i=1}^{12}x_iD_i$ on $V$ as ({\ref{Z}}). 
Noting that $D_j = \widetilde{f}_{\ast}^{-1}(\widetilde{D} _j )$ for $j=1,\dots ,10$, we have: 
\begin{align*}
x_1 = 2,\ x_2 = 4,\ x_3 = 3,\ x_j = 10-j\ (j=4,\dots ,9),\ x_{10} = x_{11} = 0,\ x_{12} = 1. 
\end{align*}
Then we see that $\Supp (Z) = \Supp (E)$ and $p_a(Z) = 1$ by a straightforward calculation, where $E$ is the connected component of $D$ with a branching component. 
\end{eg}
\section{Further problems}\label{6}
By arguments in \S \ref{4} and \S \ref{5}, we infer that minimal compactifications of $\bC ^2$ with the numerically trivial canonical divisors are restrictive. 
Hence, we establish the following problem: 
\begin{prob}
Classify the minimal compactifications of the affine plane $\bC ^2$ with numerically trivial canonical divisors. 
\end{prob}
We also present the following problem: 
\begin{prob}
Let $S$ be a $\bQ$-homology plane with logarithmic Kodaira dimension $-\infty$ such that there exists a minimal compactification $(X,\Gamma )$ of $S$ (i.e., $X$ is a normal analytic surface and $\Gamma$ is an irreducible curve on $X$ such that $X \backslash \Gamma$ is biholomorphic to $S$). 
Assume that every singular point of $X$ is rational. 
Then when is the canonical divisor $K_X$ nef (especially, numerically trivial) ?
\end{prob}
Theorem \ref{main} implies that there is no minimal compactification $(X,\Gamma )$ of $S$ such that $X$ has at worse rational singularities and $K_X$ is nef, if $S=\bC ^2$. 
Meanwhile, we find the following example. 
\begin{eg}\label{eg(6-1)}
Let $S$ be the projective plane with an irreducible conic removed. 
It is well known that $S$ is a $\bQ$-homology plane with logarithmic Kodaira dimension $-\infty$. 
Letting $Q$ be a irreducible conic on $\bP ^2_{\bC}$, we consider a compactification $(\bP ^2_{\bC},Q)$ of $S$. 
Let $\nu :V \to \bP ^2_{\bC}$ be a the birational morphism such that $\nu ^{\ast}(Q)_{\rm red.} = C+D$, where $C$ is a $(-1)$-curve and the weighted dual graph of $C+D$ is as follows: 
\begin{align*}
\xygraph{
\circ ([]!{+(0,.2)} {^{D_2}}) -[r] \circ ([]!{+(0,.2)} {^{D_3}}) -[r] \circ ([]!{+(0,.2)} {^{D_4}}) -[r] \circ ([]!{+(0,.2)} {^{D_5}}) -[r] \circ ([]!{+(0,.2)} {^{D_6}}) -[r] \circ ([]!{+(0,.2)} {^{D_7}}) (-[d] \circ ([]!{+(.3,0)} {^{D_1}}),-[r] \circ ([]!{+(0,-.3)} {^{-3}}) ([]!{+(0,.2)} {^{D_8}}) -[r] \circ ([]!{+(0,.2)} {^{D_{10}}}) (-[d] \circ ([]!{+(.3,0)} {^{D_9}}),-[r] \circ ([]!{+(0,.2)} {^{D_{11}}}) -[r] \bullet ([]!{+(0,.2)} {^C})))}
\end{align*}
Here, in the above graph, each $D_i$ is an irreducible component of $D$, and $\nu _{\ast}^{-1}(Q) = D_1$. 
By a straightforward calculation, we know that the determinant of the intersection matrix of $D$ is equal to $16$. 
Hence, by Lemma, \ref{cont} there exists a contraction $\pi :V \to X$ of $D$ such that $(X,\pi _{\ast}(C))$ is a minimal compactification of $S$. 
Moreover, for any $\bZ$-divisor $Z = \sum _{i=1}^{11}x_iD_i$ with $\Supp (Z) = \Supp (D)$, we see $p_a(Z) \le 0$ (see Appendix \ref{7}). 
This implies that $\Sing (X)$ consists of only one singular point, which is rational, by Lemma \ref{rational sing}. 
Meanwhile, since the intersection matrix of $D$ is negative definite, there exists uniquely an effective $\bQ$-divisor $D^{\sharp}=\sum _{i=1}^{11} \alpha _iD_i$ on $V$ such that $(D_i \cdot K_V+D^{\sharp})=0$ for every irreducible component $D_i$ of $D$. 
Then we have: 
\begin{align*}
\alpha _1 = \frac{3}{2},\ 
\alpha _i = \frac{i-1}{2}\ (i=2,\dots ,7),\ 
\alpha _8 =2,\ 
\alpha _9 = \alpha _{11} =1,\ 
\alpha _{10} = 2. 
\end{align*}
Hence, $(C \cdot D^{\sharp}) = \alpha _{11} = 1$. 
Thus, $K_X$ is nef; more precisely, numerically trivial (cf.\ Lemma \ref{numerical}). 
In particular, $X$ is not a numerical del Pezzo surface. 
\end{eg}
\appendix
\section{A supplement to Example \ref{eg(6-1)}}\label{7}
Let $V$ be a smooth projective surface. 
Assume that there exists a reduced divisor $D = \sum _{i=1}^{11}D_i$ on $V$ such that any $D_i$ is a smooth rational curve and the weighted dual graph of $D$ is as follows: 
\begin{align*}
\xygraph{
\circ ([]!{+(0,.2)} {^{D_2}}) -[r] \circ ([]!{+(0,.2)} {^{D_3}}) -[r] \circ ([]!{+(0,.2)} {^{D_4}}) -[r] \circ ([]!{+(0,.2)} {^{D_5}}) -[r] \circ ([]!{+(0,.2)} {^{D_6}}) -[r] \circ ([]!{+(0,.2)} {^{D_7}}) (-[d] \circ ([]!{+(.3,0)} {^{D_1}}),-[r] \circ ([]!{+(0,-.3)} {^{-3}}) ([]!{+(0,.2)} {^{D_8}}) -[r] \circ ([]!{+(0,.2)} {^{D_{10}}}) (-[d] \circ ([]!{+(.3,0)} {^{D_9}}),-[r] \circ ([]!{+(0,.2)} {^{D_{11}}}) ))}
\end{align*}
Let $Z = \sum _{i=1}^{11}x_iD_i$ be an effective $\bZ$-divisor on $V$ such that $\Supp (Z) = \Supp (D)$. 
The purpose of this appendix is to prove that $p_a(Z)$ is a non-positive integer by direct calculation. 

For simplicity, we write: 
\begin{align*}
a := -\frac{1}{2}(2x_1&-x_7)^2,\ b := -x_2^2 -\sum _{i=2}^6(x_i-x_{i+1})^2,\ c := -\frac{1}{2}(x_7-2x_8)^2,\ \\
&d := -(x_8-x_{10})^2 -\frac{1}{2}(2x_9-x_{10})^2 -\frac{1}{2}(2x_{11}-x_{10})^2. 
\end{align*}
By the definitions, we note as follows: 
\begin{align*}
(Z \cdot K_V+Z) = a+b+c+d+x_8,\ a \le 0,\ b \le -x_2^2,\ c\le 0,\ d \le 0. 
\end{align*}

In order to show $p_a(Z) \le 0$, we shall prove $(Z \cdot K_V+Z) \le -2$. 
Suppose that $x_8 \ge 3$. 
Since: 
\begin{align*}
b+c + x_8 = -\sum _{i=2}^6\frac{i}{i-1}\left(x_i-\frac{i-1}{i}x_{i+1}\right)^2 -\frac{1}{6}(2x_7-3x_8)^2 + \frac{1}{2}x_8(2-x_8) \le \frac{1}{2}x_8(2-x_8)
\end{align*}
and $x_8 \ge 3$, we know $(Z \cdot K_V+Z) \le \frac{1}{2}x_8(2-x_8) \le -\frac{3}{2}$. 
This implies that $(Z \cdot K_V+Z) \le -2$ because $(Z \cdot K_V+Z)$ is an integer. 
Hence, we assume $x_8 \le 2$ in what follows. 
Suppose that $x_2 \ge 2$. 
Then we have $(Z \cdot K_V+Z) \le b + x_8 \le -x_2^2 + x_8 \le -2$. 
Hence, we further assume $x_2 = 1$. 
Then we notice: 
\begin{align*}
b = -x_2^2 - \sum _{i=2}^6(x_i-x_{i+1})^2 \le -1^2 -(x_7-1) = -x_7. 
\end{align*}
Now, we consider the two cases $x_8=2$ and $x_8=1$, separately. 

Assume that $x_8=2$. 
If $x_8=2$ and $x_7 \ge 4$, then we have $(Z \cdot K_V+Z) \le b +x_8 \le -4 + 2 = -2$. 
In what follows, we consider the case $x_7 \le 3$. 
Then we have $(Z \cdot K_V+Z) \le b+c +x_8 \le -x_7+c+x_8 \le -\frac{3}{2}$. 
This implies that $(Z \cdot K_V+Z) \le -2$ because $(Z \cdot K_V+Z)$ is an integer. 

Finally, assume that $x_8=1$. Then we note $d \le -1$. 
If $x_8=1$ and $x_7 \ge 2$, then we have $(Z \cdot K_V+Z) \le b+d +x_8 \le -2 -1 +1 = -2$. 
If $x_8=1$ and $x_7=1$, then we have $(Z \cdot K_V+Z) = a+b+c+d+x_8 \le -\frac{1}{2}(2x_1-1)^2 -x_7 -\frac{1}{2}(1-2)^2 -1 +x_8 \le -\frac{1}{2} -1 -\frac{1}{2} -1 + 1 = -2$. 

Hence, we thus obtain $p_a(Z) \le 0$. 

\end{document}